\documentclass[11pt]{article}
\usepackage{definitions}

\newcommand{\one}{\mathbf{1}}
\newcommand{\I}{\mathrm{I}}
\newcommand{\vvec}{\mathrm{vec}}
\newcommand{\subA}{\hat{A}}
    
\begin{document}

\title{PDOT: a Practical Primal-Dual Algorithm and a GPU-Based Solver for Optimal Transport}
\author{Haihao Lu\thanks{MIT, Sloan School of Management (haihao@mit.edu).} \and Jinwen Yang\thanks{University of Chicago, Department of Statistics (jinweny@uchicago.edu).}}
\date{}

\maketitle
\begin{abstract}
    In this paper, we propose a practical primal-dual algorithm with theoretical guarantees and develop a GPU-based solver, which we dub PDOT, for solving large-scale optimal transport problems. Compared to Sinkhorn algorithm or classic LP algorithms, PDOT can achieve high-accuracy solution while efficiently taking advantage of modern computing architecture, i.e., GPUs. On the theoretical side, we show that PDOT has a data-independent $\widetilde O(mn(m+n)^{1.5}\log(\frac{1}{\epsilon}))$ local flop complexity where $\epsilon$ is the desired accuracy, and $m$ and $n$ are the dimension of the original and target distribution, respectively. We further present a data-dependent $\widetilde O(mn(m+n)^{3.5}\Delta + mn(m+n)^{1.5}\log(\frac{1}{\epsilon}))$ global flop complexity of PDOT, where $\Delta$ is the precision of the data. On the numerical side, we present PDOT, an open-source GPU solver based on the proposed algorithm. Our extensive numerical experiments consistently demonstrate the well balance of PDOT in computing efficiency and accuracy of the solution, compared to Gurobi and Sinkhorn algorithms.
\end{abstract}
\section{Introduction}
Optimal transport (OT), also known as mass transportation, the earth mover's distance, and the Wasserstein-1 ($W_{1}$) distance, represents a crucial class of optimization problems that quantify the distance between probability distributions~\cite{villani2009optimal,villani2021topics}. Originating in the 18th century through Monge's pioneering work~\cite{monge1781founding}, OT has found extensive applications across various fields including operations research~\cite{blanchet2023unifying,tacskesen2023semi,shafieezadeh2023new,blanchet2022optimal,si2021testing}, economics~\cite{galichon2018optimal,fajgelbaum2020optimal,daskalakis2013mechanism,torous2021optimal}, machine learning~\cite{courty2014domain,courty2016optimal,xu2019gromov,chen2019improving}, computer vision and image processing~\cite{rabin2014adaptive,ge2021ota,liu2020semantic,solomon2014earth,rubner2000earth}, biology~\cite{schiebinger2019optimal,moriel2021novosparc,huizing2022optimal,demetci2020gromov,cao2022unified}, and quantum mechanics~\cite{ikeda2020foundation,santambrogio2015optimal}, among others.

In this paper, we focus on solving the discrete optimal transport (OT) problem, initially formulated by Kantorovich~\cite{kantorovich1942translocation}. The discrete OT problem is mathematically expressed as:
\begin{equation}\label{eq:ot}
    \begin{aligned}
        \min_{X \geq 0} & \ \langle C, X \rangle \\ 
        \mathrm{s.t.} & \ X \mathbf{1}_n = f \\ 
        & \ X^\top \mathbf{1}_m = g \ ,
    \end{aligned}
\end{equation}
where $C = [C_{ij}]_{1 \leq i \leq m, 1 \leq j \leq n}$ denotes the given non-negative cost matrix, and $f = [f_i]_{i=1}^m$ and $g = [g_j]_{j=1}^n$ are probability vectors representing the row and column marginals, respectively. Here, $\mathbf{1}_n$ and $\mathbf{1}_m$ are vectors of ones with dimensions $n$ and $m$. The objective is to find a non-negative matrix $X$, also called transportation plan in optimal transport, that minimizes the total transportation cost $\langle C, X \rangle$, subject to the constraints that the marginals of $X$ match the given vectors $f$ and $g$ that are in general distributions.

The current practical approaches to solving optimal transport (OT) problems encompass a range of classical and modern methods. One classical approach that is particularly suitable for OT is the network simplex algorithm, by noticing OT is a network flow problem. The network simplex algorithm, an adaptation of the simplex method for network flow problems, is robust and provides exact solutions~\cite{orlin1997polynomial}. This algorithm operates by exploiting the network structure of the OT problem, iterating through potential solutions by traversing the edges of the feasible region's polytope. It is particularly effective for small-to-medium scale problems but can be computationally intensive and slow for very large datasets. Moreover, the necessity of frequent linear system solving, as well as corresponding matrix factorization, impedes its efficient implementation on GPU. In contrast, first-order methods such as the Sinkhorn algorithm (also known as matrix scaling) have gained popularity due to their scalability and efficiency. The Sinkhorn algorithm approximates the OT problem by adding an entropy regularization term, making the problem easier to solve using iterative matrix scaling~\cite{cuturi2013sinkhorn}. This method is highly parallelizable, which makes it suitable for large-scale applications as in machine learning and image processing. However, the trade-off is that the approximation depends on the chosen penalty parameter. This incurs at least two well-known drawbacks of Sinkhorn: (i) Sinkhorn is numerically unstable, even for relatively small penalty parameter. The numerical issue can be alleviated to some extent by stabilization techniques recently developed in \cite{schmitzer2019stabilized}. (ii) Sinkhorn often provides low-accuracy solution to original optimal transport problems. An extremely small penalty parameter is required to achieve high-accuracy, while this approach faces numerical issue and significantly slowdowns the convergence. Overall, while classical methods offer precise solutions, their computational cost and lack of efficient parallel implementation can be prohibitive for large datasets, whereas Sinkhorn method provides a more practical efficiency for large-scale OT problems with the cost of providing low-quality solutions.

\begin{table}
\centering
\includegraphics[width=0.7\textwidth]{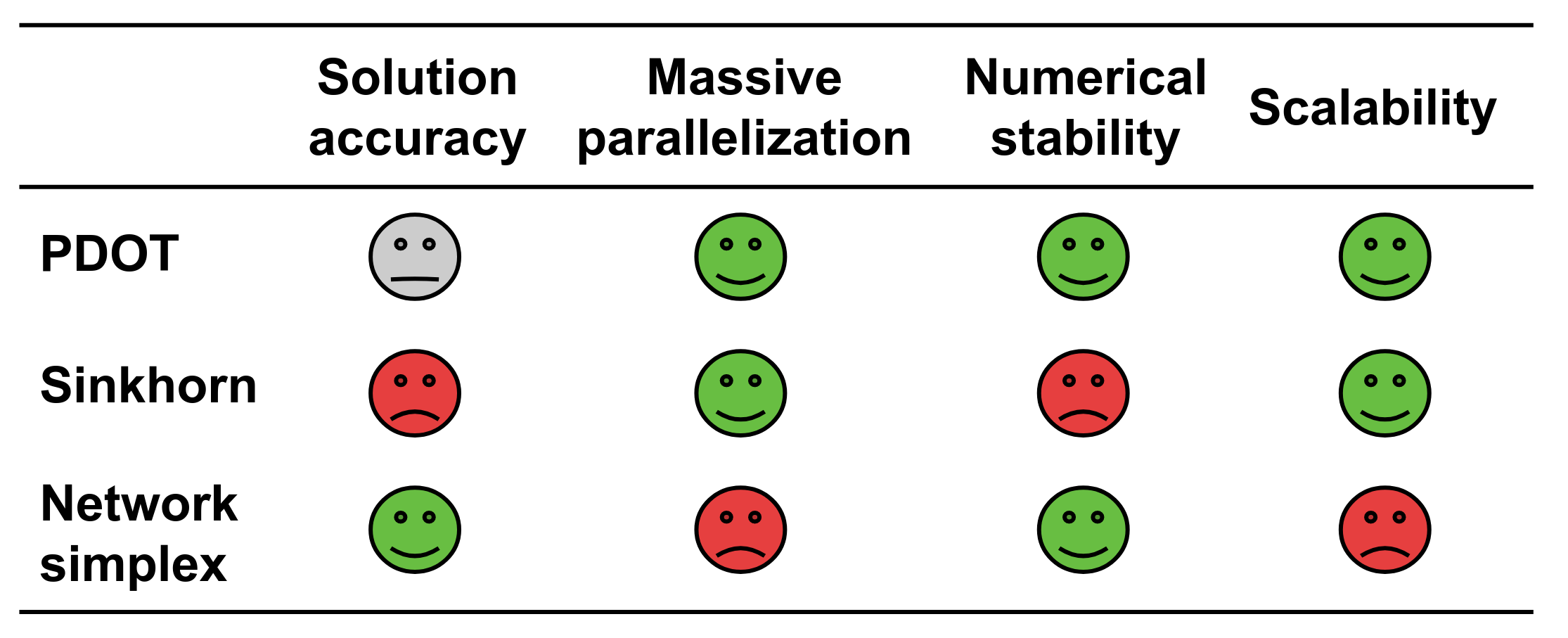}
\caption{Comparison of PDOT, Sinkhorn and network simplex in four dimensions, where green, gray, and red faces represent a decreasing order of favorableness of the corresponding method in the corresponding dimension.} 
\label{tab:compare} 
\end{table}

This paper aims to address the following natural questions:
\vspace{0.2cm}
\begin{center}
    \textit{Is there a practical algorithm for large-scale OT problems that can achieve high-accuracy solutions and be efficiently implemented on GPUs?}
\end{center}
\vspace{0.2cm}

We provide an affirmative answer to the above question by presenting PDOT, a practical primal-dual algorithm and a GPU-based solver for large-scale OT problems. Table \ref{tab:compare} summarizes the comparison of PDOT, Sinkhorn and network simplex method in four different dimensions: solution accuracy,  whether they can take advantage of massive parallelization, and whether they are numerically stable and scalability. Generally speaking, PDOT inherits the advantage of Sinkhorn (since the computational bottleneck for both algorithms are matrix-vector multiplications) while is able to achieve higher-accuracy solutions.

The contributions of the work can be summarized as follows:
\begin{itemize}
    \item We propose using restarted PDHG for solving large-scale optimal transport problems (Section \ref{sec:rpdhg}). The algorithm is adapted to the specific structure of OT problems, leading to an efficient implementation on GPUs without constructing the huge constrained matrix.
    \item We show that restarted PDHG exhibits fast local linear convergence after identification (Section \ref{sec:local-rate}, with total flop count $\widetilde O(mn(m+n)^{1.5}\log\frac{1}{\epsilon})$ while the iteration complexity and parallel depth\footnote{Parallel depth, also known as span in parallel computing, represents how much of our algorithm must still happen in serial even with infinite processors~\cite{blelloch1996programming}. This concept has been used in analyzing the flop complexity of OT algorithms.}
    are merely $\widetilde O((m+n)^{1.5}\log\frac{1}{\epsilon})$. We further establish $\widetilde O((m+n)^{3.5}\Delta+(m+n)^{1.5}\log\frac{1}{\epsilon})$ global iteration complexity and parallel depth when the data are rational numbers with $\widetilde O(mn(m+n)^{3.5}\Delta+mn(m+n)^{1.5}\log\frac{1}{\epsilon})$ total flop count, where $\Delta$ is the precision of the data (Section \ref{sec:global-rate}).
    \item We develop PDOT, a GPU-implemented solver based on restarted PDHG, for solving large-scale optimal transport. Extensive numerical experiments demonstrate the efficiency and effectiveness of the solver compared to Sinkhorn algorithm and state-of-the-art LP solvers (Section \ref{sec:numerical}).
\end{itemize}

The development of PDOT is part of a broader trend towards matrix-free solvers that employ first-order methods (FOMs) for large-scale optimization problems, outpacing the capabilities of traditional methods such as simplex and barrier algorithms. Several matrix-free solvers have been developed in this research area. Examples include \href{https://github.com/google/or-tools/tree/stable/ortools/pdlp}{PDLP}/cuPDLP (\href{https://github.com/jinwen-yang/cuPDLP.jl}{cuPDLP.jl}, \href{https://github.com/COPT-Public/cuPDLP-C}{cuPDLP-C})~\cite{applegate2021practical,lu2023cupdlp,lu2023cupdlpc}, which are based on the primal-dual hybrid gradient (PDHG) method, and the matrix-free interior point method (IPM) solver \href{https://github.com/leavesgrp/ABIP}{ABIP}~\cite{lin2021admm,deng2022new}. Other notable solvers include the dual-based LP solver ECLIPSE~\cite{basu2020eclipse}, and general conic solver \href{https://github.com/cvxgrp/scs}{SCS}~\cite{o2016conic,o2021operator}, and quadratic programming solver \href{https://github.com/osqp/osqp}{OSQP}~\cite{stellato2020osqp}, which employ the alternating direction method of multipliers (ADMM). Efficient implementation of FOM-based solvers on GPUs has demonstrated performance comparable to state-of-the-art commercial solvers on standard benchmark sets, achieving high-quality solutions within reasonable time~\cite{lu2023cupdlp,lu2023cupdlpc}. This highlights the potential and effectiveness of FOM-based approaches in addressing large-scale problems.


\subsection{Related literature}

{\bf Classic methods for optimal transport.} Algorithms for solving discrete optimal transport problems are fundamental in various applications that requires efficient methods to determine optimal transport plans. Network simplex method is a powerful adaptation of the simplex method for network flow problems to find the optimal flow that minimizes the cost~\cite{papadimitriou2013combinatorial,ahuja1993network,orlin1997polynomial}. Hungarian algorithm, also known as Kuhn-Munkres algorithm, addresses OT as an assignment problem for bipartite graphs~ \cite{kuhn1955hungarian,munkres1957algorithms}. Auction algorithm, proposed in \cite{bertsekas1992auction, bertsekas1988auction}, offers an iterative bidding approach to the assignment problem. More recently, a variant of interior-point method is proposed to effectively solve LPs arising in OT~\cite{zanetti2023interior}. We refer to \cite[Chapter 3]{peyre2019computational} for more detailed discussions on classic approaches on discrete OT.

{\bf First-order algorithms for optimal transport.} First-order methods (FOMs) have proven to be highly effective for solving large-scale optimal transport problems. FOMs are highly parallelizable and enjoy low per-iteration cost, significantly reducing computational time for large datasets, although it produces approximate rather than exact solutions. The Sinkhorn algorithm, originally proposed by Sinkhorn \cite{sinkhorn1967diagonal}, introduces entropic regularization to the optimal transport problem, making it more tractable through iterative matrix scaling and allowing for fast computation of transport plans \cite{cuturi2013sinkhorn}. Later on in \cite{mishchenko2019sinkhorn,leger2021gradient,aubin2022mirror}, it is shown that Sinkhorn can be interpreted as mirror descent on the dual problem under relative smoothness~\cite{lu2018relatively}. Empirically it is well-known that vanilla Sinkhorn has numerical issues with small regularization parameters, and variants of Sinkhorn are proposed to improve its numerical stability~\cite{schmitzer2019stabilized} and boost its practical performance~\cite{altschuler2017near,scetbon2021low}. Furthermore, many existing works apply other first-order methods for solving discrete OT problems, such as accelerated gradient descent, stochastic gradient descent, iterative Bregman method, mirror descent and operator splitting~\cite{dvurechensky2018computational, genevay2016stochastic,benamou2015iterative,lin2019efficient,mai2021fast,chambolle2022accelerated}. Several OT solvers have already been off-the-shelf with various algorithmic options, such as \href{https://pythonot.github.io/}{POT} in Python~\cite{flamary2021pot}, \href{https://ott-jax.readthedocs.io/en/latest/}{OTT} in JAX~\cite{cuturi2022optimal}, \href{https://github.com/JuliaOptimalTransport/OptimalTransport.jl}{OptimalTransport.jl} in Julia, etc. For more detailed discussions, we refer to \cite[Chapter 4]{peyre2019computational}.

Another active line of research focuses on improving the theoretical complexity of solving OT, see for example~\cite{altschuler2017near,dvurechensky2018computational,blanchet2024towards,jambulapati2019direct,lin2022efficiency,quanrud2018approximating,van2020bipartite}. While we also present theoretical complexity result of our proposed algorithm, we emphasis that we aim to develop a practical algorithm with strong numerical performance and theoretical guarantees instead of improving the theoretical bound on existing literature.

{\bf FOM solver for linear programming.}
Recent researches have increasingly focused on the application of first-order methods (FOMs) for solving large-scale linear programming problems. This interest is driven by the low iteration costs and the potential for parallelization offered by these methods. Particularly, PDLP is a general-purpose large-scale LP solver that builds on the restarted PDHG algorithm~\cite{applegate2023faster} and incorporates numerous practical enhancements~\cite{applegate2021practical,lu2023cupdlp,lu2023cupdlpc}. The CPU implementation of PDLP, which is open-sourced via \href{https://developers.google.com/optimization}{Google OR-Tools}, has been shown to outperform other general-purpose FOM solvers for LPs~\cite{applegate2021practical}. Additionally, the GPU implementation, cuPDLP (\href{https://github.com/jinwen-yang/cuPDLP.jl}{cuPDLP.jl} and \href{https://github.com/COPT-Public/cuPDLP-C}{cuPDLP-C}), has demonstrated performance on par with commercial LP solvers such as Gurobi and COPT~\cite{lu2023cupdlp,lu2023cupdlpc}. See \cite{lu2024first} for a review on the recent development of first-order methods for LP.

\subsection{Notation}
Denote $\mathcal Z^*$ the primal-dual optimal solution set. Let $\|\cdot\|_2$ be the Euclidean norm for vectors and spectral norm for matrices, and $\|\cdot\|_F$ the Frobenius norm of matrices. $\|\cdot\|_1$ and $\|\cdot\|_\infty$ are the $\ell_1$ and $\ell_\infty$ norm of vectors respectively. The (Euclidean) distance between a point $v$ and a set $\mathcal U$ is defined as $\mathrm{dist}(v,\mathcal U):=\min_{u\in \mathcal U}\|v-u\|_2$. Define $[n]:=\{1,...,n\}$. For a vector $u$, denote $(u)_i$ its $i$-th element and $(u)_S:=((u)_i)_{i\in S}$ for index set $S$. We use big O notation to characterize functions according to their growth rate, particularly, $f(x)= O(g(x))$ means that for sufficiently large $x$, there exists constant $C$ such that $f(x)\leq Cg(x)$ while $f(x)=\widetilde O(g(x))$ suppresses the log dependence, i.e., $f(x)=O(g(x)\log(g(x))$.

\section{Preliminary}
In this section, we briefly revisit the current development of restarted PDHG for solving linear programming. Specifically, consider the standard linear programming
\begin{equation}\label{eq:standard-lp}
    \min_{x\geq 0}\ c^\top x\quad \mathrm{s.t.}\ Ax=b \ ,
\end{equation}
with primal-dual formulation
\begin{equation}\label{eq:standard-lp-minimax}
    \min_{x\geq 0}\max_{y}\ c^\top x-y^\top Ax+b^\top y \ .
\end{equation}
Algorithm \ref{alg:pdhg-lp} below demonstrates how restarted PDHG solves \eqref{eq:standard-lp-minimax}. For each outer loop $t$, we run vanilla PDHG on \eqref{eq:standard-lp-minimax} until a restart condition holds. When restart of outer loop happens, the initial solution is reset at the average iterates of last inner loop. 
\begin{algorithm}[ht!]
        \SetAlgoLined
        {\bf Input:} Initial point $(x^{0,0},y^{0,0})$, step-size $\eta, \sigma$\;
            \Repeat{\textit{$(x^{t,0}, y^{t,0})$ converges}}{
                {\bf initialize the inner loop.} inner loop counter $k \gets 0$\;

            \Repeat{a restart condition holds}{
                $x^{t,k+1}\gets \mathrm{Proj}_{\mathbb R_+^{n}}(x^{t,k}-\eta(c-A^\top y^{t,k}))$\;
                \vspace{0.075cm}
                $y^{t,k+1}\gets y^{t,k}+\sigma(b-A(2x^{t,k+1}-x^{t,k}))$\;
                \vspace{0.075cm}
                $(\bar x^{t,k+1},\bar y^{t,k+1})\gets \frac{k}{k+1}(\bar x^{t,k},\bar y^{t,k})+\frac{1}{k+1}(x^{t,k+1},y^{t,k+1})$\;
            }
        {\bf restart the outer loop.} $(x^{n+1,0}, y^{n+1,0})\gets (\bar x^{t,k},\bar y^{t,k})$, $n\gets n+1$\;
            }
        {\bf Output:} $(x^{t,0}, y^{t,0})$.
        \caption{Restarted PDHG on \eqref{eq:standard-lp}}
        \label{alg:pdhg-lp}
\end{algorithm}
We adopt the restart scheme with respect to KKT error of \eqref{eq:standard-lp}. Specifically, KKT error at solution $z=(x,y)$, with $\|z\|_2\leq R$,  is defined as $\mathrm{KKT}(z):=\mathrm{KKT}(x,y)=\left\|\begin{bmatrix}
        Ax-b \\ [A^\top y-c]^+ \\ \frac 1R[c^\top x-b^\top y]
    \end{bmatrix}\right\|_2$ which measures the primal infeasibility, dual infeasibility and duality gap. We restart the outer loop whenever the KKT error at running average has sufficient decay $\beta\in(0,1)$ from the initial KKT:
    \begin{equation}\label{eq:restart-condition-lp}
        \mathrm{KKT}(\bar x^{t,k},\bar y^{t,k})\leq \beta \mathrm{KKT}(x^{t,0},y^{t,0}) \ .
    \end{equation}

It turns out that restarted PDHG with restart scheme \eqref{eq:restart-condition-lp} exhibits strong theoretical guarantee for solving linear programming. We restate the convergence guarantee of restarted PDHG in the lemma below that establishes the linear rate to achieve $\epsilon$-accuracy. Without loss of generality, assume $\sigma=\eta$ throughout the rest of the paper, that is, we assume the primal and dual step-sizes are equal (this can be achieved without the loss of generality by rescaling the variables, see \cite{applegate2023faster} for a discussion).
\begin{lem}[{\cite[Theorem 1]{lu2023cupdlp}}]\label{lem:basic}
    Consider $\{z^{t,k}=(x^{t,k},y^{t,k})\}$ the restarted PDHG iterates with restart scheme \eqref{eq:restart-condition-lp} using decay $\beta\in(0,1)$ and step-size $\eta\leq \frac{1}{2\|A\|_2}$ for solving \eqref{eq:standard-lp-minimax}, where $R=2\|z^{0,0}-z^*\|_2+2\|z^*\|_2$ for one $z^*\in \mathcal Z^*$. Denote $\alpha>0$ the sharpness constant of the KKT system of \eqref{eq:standard-lp} on $B_R(0)$, namely
    \begin{equation*}
        \alpha\mathrm{dist}(z, \mathcal Z^*)\leq \mathrm{KKT}(z), \quad \forall z\in B_R(0) \ .
    \end{equation*}
    Then it holds that
    \begin{enumerate}
        \item[(i)] There exists an optimal solution $z^*=(x^*,y^*)$ such that $\lim_{t\rightarrow \infty}z^{t,0}=z^*$, and $\|z^{t,k}\|_2\leq R$ for any $t$ and $k$;
        \item[(ii)] The restart frequency $\tau^t\leq \frac{16\|A\|_2}{\beta\alpha}$; 
        \item[(iii)] The linear convergence holds: $\mathrm{dist}(z^{t,0},\mathcal Z^*)\leq \beta^{t}\frac{1}{\alpha}\mathrm{KKT}(z^{0,0})$.
    \end{enumerate}
\end{lem}

\section{Restarted PDHG for optimal transport}\label{sec:rpdhg}

In this section, we demonstrate how to efficiently implement restarted PDHG (Algorithm \ref{alg:pdhg-lp}) for optimal transport problems \eqref{eq:ot} by utilizing its problem structure, and how to obtain a strictly feasible solution by rounding after a near feasible solution is obtained.

\subsection{Efficient implementation}
Optimal transport \eqref{eq:ot} can be recast as a standard form LP \eqref{eq:standard-lp} with
\begin{equation}\label{eq:ot-lp-data}
    c=\vvec(C)\in\mathbb R^{mn},\ b=\begin{bmatrix}
    f \\ g
\end{bmatrix}\in\mathbb R^{m+n},\  A=\begin{bmatrix}
        \one_n^\top \otimes \I_m \\ \I_n \otimes \one_m^\top
    \end{bmatrix}\in \mathbb R^{(m+n)\times mn} \ , 
\end{equation}
where $\otimes$ is the Kronecker product.
However, a direct implementation of Algorithm \ref{alg:pdhg-lp} requires explicitly formulating the huge constrained matrix $A\in R^{(m+n)\times mn}$, which can be memory-intensive and computationally expensive. To avoid this, we here propose an efficient implementation by utilizing the Kronecker product structure of the constrained matrix $A$. Specifically, we notice that the two matrix-vector multiplication steps in Algorithm \ref{alg:pdhg-lp} can be computed as  $Ax=\begin{bmatrix}
        X\one_n \\ X^\top \one_m
    \end{bmatrix}$ for any $x=\vvec(X)$, and $A^\top\begin{bmatrix}
        p \\ q
    \end{bmatrix}=p\one_n^\top+\one_mq^\top$ for vector $p$ and $q$. Inspired by such an efficient computation scheme, we introduce Algorithm \ref{alg:rpd} for solving optimal transport. Here $X^{t,k}$ represents the primal iterates (i.e., transportation plan), and $(p^{t,k},q^{t,k})$ stand for the dual iterates corresponding to row and column marginals respectively.

\begin{algorithm}[ht!]
        \SetAlgoLined
        {\bf Input:} Initial point $(X^{0,0},p^{0,0},q^{0,0})$, step-size $\eta, \sigma$\;
            \Repeat{\textit{$(X^{t,0}, p^{t,0}, q^{t,0})$ converges}}{
                {\bf initialize the inner loop.} inner loop counter $k \gets 0$\;

            \Repeat{a restart condition holds}{
                $X^{t,k+1}\gets \mathrm{Proj}_{\mathbb R_+^{m\times n}}(X^{t,k}-\eta(C-p^{t,k}\one_n^\top-\one_m {q^{t,k}}^\top))$\;
                \vspace{0.075cm}
                $p^{t,k+1}\gets p^{t,k}+\sigma(f-(2X^{t,k+1}-X^{t,k})\one_n)$\;
                \vspace{0.075cm}
                $q^{t,k+1}\gets q^{t,k}+\sigma(g-(2X^{t,k+1}-X^{t,k})^\top\one_m)$\;
                \vspace{0.075cm}
                $(\bar X^{t,k+1},\bar p^{t,k+1},\bar q^{t,k+1})\gets \frac{k}{k+1}(\bar X^{t,k},\bar p^{t,k},\bar q^{t,k})+\frac{1}{k+1}(X^{t,k+1},p^{t,k+1},q^{t,k+1})$\;
            }
        {\bf restart the outer loop.} $(X^{n+1,0}, p^{n+1,0}, q^{n+1,0})\gets (\bar X^{t,k},\bar p^{t,k},\bar q^{t,k})$, $n\gets n+1$\;
            }
        {\bf Output:} $(X^{t,0}, p^{t,0}, q^{t,0})$.
        \caption{Efficient implementation of restarted PDHG for optimal transport \eqref{eq:ot}}
        \label{alg:rpd}
\end{algorithm}
Similarly, the KKT error of optimal transport \eqref{eq:ot} can be efficiently evaluated as follows without constructing the constrained matrix $A$:
\begin{equation}
    \mathrm{KKT}(X,p,q)=\left\|\begin{bmatrix}
        X\one_n-f \\ 
        X^\top \one_m-g\\
        \vvec([p\one_n^\top+\one_m q^\top-C]^+) \\ \frac1R[\langle C,X\rangle-f^\top p-g^\top q]
    \end{bmatrix}\right\|_2 \ ,
\end{equation}
where $[W]^+=[\max\{W_{ij},0\}]_{1\leq i\leq m,1\leq j\leq n}$ is the element-wise positive part and $R$ is an upper bound on $\|(X^{t,k},p^{t,k},q^{t,k})\|_2$. As restart condition for LP \eqref{eq:restart-condition-lp}, we restart the algorithm when
\begin{equation}\label{eq:restart-condition}
    \mathrm{KKT}(\bar X^{t,k},\bar p^{t,k},\bar q^{t,k}) \leq \beta\mathrm{KKT}(X^{t,0},p^{t,0},q^{t,0}) \ .
\end{equation}

This efficient implementation is critical for Algorithm \ref{alg:rpd} to achieve strong numerical performance, particularly on GPU architecture, compared to a direct implementation of Algorithm \ref{alg:pdhg-lp}. Besides the memory-efficiency without constructing the huge sparse matrix $A$, the dense matrix operations in Algorithm \ref{alg:rpd} can also be more efficient on GPUs compared to the sparse sparse matrix-vector multiplication operations in a direct implementation of Algorithm \ref{alg:pdhg-lp}.

\subsection{Post-processing for obtaining feasible transport plan}
A feasible transportation plan is often favorable in real-world applications. Nevertheless, the output of Algorithm \ref{alg:rpd} is not guaranteed to lie in the feasible region, namely the constraints on marginal typically hold within a given tolerance. To obtain a feasible plan from the approximate solution of Algorithm \ref{alg:rpd}, we leverage the rounding algorithm (Algorithm \ref{alg:round}) as post-processing. Algorithm \ref{alg:round} is simple and parallelizable and is only called once after Algorithm \ref{alg:rpd} terminates.
\begin{algorithm}[ht!]
        \SetAlgoLined
        {\bf Input:} Approximately feasible solution $X$\;
            $D_x\gets\mathrm{Diag}(x)$ with $x_i=\min\left\{\tfrac{f_i}{(X\one_n)_i}, 1\right\}$\;
            \vspace{0.075cm}
            $X_\mathrm{tmp}\gets D_xX$\;
            \vspace{0.075cm}
            $D_y\gets\mathrm{Diag}(y)$ with $y_j=\min\left\{\tfrac{g_j}{({X_\mathrm{tmp}}^\top\one_m)_j}, 1\right\}$\;
            \vspace{0.075cm}
            $X_\mathrm{tmp}\gets X_\mathrm{tmp} D_y$\;
            \vspace{0.075cm}
            $\mathrm{err}_r\gets f-X_\mathrm{tmp}\one_n$, $\mathrm{err}_c\gets g-{X_\mathrm{tmp}}^\top\one_m$\;
            \vspace{0.075cm}
            $X_{\mathrm{feas}}\gets X_\mathrm{tmp}+\mathrm{err}_r\mathrm{err}_c^\top/\|\mathrm{err}_r\|_1$
            
        {\bf Output:} $X_\mathrm{feas}$.
        \caption{Rounding~\cite[Algorithm 2]{altschuler2017near}}
        \label{alg:round}
\end{algorithm}

It turns out Algorithm \ref{alg:round} exhibits strong theoretical guarantee. The following lemma shows that the distance between the output of Algorithm \ref{alg:round} and input solution can be upper bound by input's violation of marginal constraints.
\begin{lem}[{\cite[Lemma 7]{altschuler2017near}}]\label{lem:round}
    It takes $O(mn)$ time for Algorithm \ref{alg:round} to output a feasible transportation plan $X_{\mathrm{feas}}$ such that
    \begin{equation*}
        \|X_{\mathrm{feas}}-X\|_1\leq 2(\| f-X\one_n \|_1+\|g-{X}^\top\one_m\|_1) \ .
    \end{equation*}
\end{lem}

\section{Theoretical guarantees}\label{sec:theory}

In this section, we establish theoretical guarantee for restarted PDHG (Algorithm \ref{alg:rpd}). In particular, we develop a problem-independent local linear convergence rate in Section \ref{sec:local-rate}. Furthermore, global problem-independent linear complexity is derived when the data are rational numbers in Section \ref{sec:global-rate}. {We highlight here that the results presented herein is not limited to OT, and can be directly applied to the more general totally uni-modular LP.}

\subsection{Local convergence}\label{sec:local-rate}

\begin{mydef}\label{def:partition}
For optimal transport problem \eqref{eq:ot} and an optimal primal-dual solution $(X^*,p^*,q^*)$, we define the index partition $(N,B_1,B_2)$ of primal variable elements as
    \begin{align*}
     \ N&=\{(i,j)\in [m]\times[n]: C_{ij}-p_i^*-q_j^*>0\} \\ \ B_1&=\{(i,j)\in [m]\times[n]: C_{ij}-p_i^*-q_j^*=0,\ X_{ij}^*>0\} \\ \  B_2&=\{(i,j)\in [m]\times[n]: C_{ij}-p_i^*-q_j^*=0, \ X_{ij}^*=0\} \ .
\end{align*}
Furthermore, we define the non-degeneracy metric $\delta$ as:
\begin{equation*}
    \delta := 
    \min\left\{\min_{(i,j)\in N} \frac{C_{ij}-p_i^*-q_j^*}{\sqrt{m+n}},\;\min_{(i,j)\in B_1} X_{ij}^*\right\} \ .
\end{equation*}
\end{mydef}
The following theorem builds the local linear convergence of Algorithm \ref{alg:rpd}. 
\begin{thm}\label{thm:local-rate}
    Consider $\{(X^{t,k},p^{t,k},q^{t,k})\}$ the iterates of Algorithm \ref{alg:rpd} with restart scheme \eqref{eq:restart-condition} for solving \eqref{eq:ot}. Suppose step-size $\eta\leq \frac{1}{2\sqrt{m+n}}$. Denote $z^*=(X^*,p^*,q^*)$ the converging optimal solution, namely $(X^{t,0},p^{t,0},q^{t,0})\rightarrow (X^*,p^*,q^*)$ and denote $(N, B_1, B_2)$ the partition using Definition \ref{def:partition}. Then it holds that
    \begin{enumerate}
        \item[(i)] There exists an integer $T>0$ such that for any $t\geq T$,
        \begin{equation}\label{eq:iden-N}
            X_{ij}^{t,0}=0,\ C_{ij}-p_i^{t,0}-{q_j^{t,0}}>0\ , \quad \forall\ (i,j)\in N\ ,
        \end{equation}
        and
        \begin{equation}\label{eq:iden-B}
            X_{ij}^{t,0}>0\ , \quad \forall\ (i,j)\in B_1\ .
        \end{equation}
        \item[(ii)] It holds for any $t\geq T$ that
        \begin{equation*}
        \tau^t\leq \frac{16}{\beta}(m+n)^{1.5} \ ,
        \end{equation*}
        and
        \begin{equation*}
        \mathrm{dist}((X^{t,0},p^{t,0},q^{t,0}),\mathcal Z^*)\leq \beta^{t-T}(m+n)\mathrm{KKT}(X^{T,0},p^{T,0},q^{T,0}) \ .
    \end{equation*}
    \end{enumerate}
\end{thm}

Theorem \ref{thm:local-rate} implies the finite-time identification and the local convergence of Algorithm \ref{alg:rpd}. Part (i) shows that identification of set $N$ and $B_1$, i.e., \eqref{eq:iden-N} and \eqref{eq:iden-B} hold, can be achieved in finite time, while part (ii) establishes the local linear convergence after identification.

\begin{rem}
    Theorem \ref{thm:local-rate} implies $\widetilde O\pran{(m+n)^{1.5}\log\frac{1}{\epsilon}}$ local iteration complexity to find an $\epsilon$-accuracy solution. Since each iteration of Algorithm \ref{alg:rpd} requires $O(mn)$ operations, the local flop complexity is $\widetilde O\pran{mn(m+n)^{1.5}\log\frac{1}{\epsilon}}$. In comparison, the flop count of network simplex equals $\widetilde O\pran{mn(m+n)\log\frac{1}{\epsilon}}$, which is better than Algorithm \ref{alg:rpd} by the order of $\widetilde O(\sqrt{m+n})$. However, it is straightforward to implement Algorithm \ref{alg:rpd} in parallel. In other words, Algorithm \ref{alg:rpd} can achieve $\widetilde O\pran{(m+n)^{1.5}\log\frac{1}{\epsilon}}$ parallel depth.
    Furthermore, in terms of Lemma \ref{lem:round}, the cost of rounding (Algorithm \ref{alg:round}) is merely logarithmic. Thus
    Algorithm \ref{alg:rpd} with rounding (Algorithm \ref{alg:round}) to find a feasible $\epsilon$-accuracy solution has iteration complexity $\widetilde O\pran{(m+n)^{1.5}\log\frac{1}{\epsilon}}$, with flop complexity $\widetilde O\pran{mn(m+n)^{1.5}\log\frac{1}{\epsilon}}$ and parallel depth $\widetilde O\pran{(m+n)^{1.5}\log\frac{1}{\epsilon}}$.
\end{rem}

The rest of this section develops the proof of Theorem \ref{thm:local-rate}. 
We will show that the highly structured constraint matrix $A$ in \eqref{eq:ot-lp-data} enables faster convergence of Algorithm \ref{alg:rpd}. In the following proposition, we first present two basic spectral properties of constraint matrix $A$.
\begin{prop}\label{prop:norm-sharp}
Denote $\subA$ a non-singular sub-matrix of matrix $A$. Then it holds that
    \begin{enumerate}
        \item[(i)] $\|\subA\|_2\leq \sqrt{m+n}$,
        \item[(ii)] $\|\subA^{-1}\|_2\leq m+n$.
    \end{enumerate}
\end{prop}
\begin{proof}
    (i) We derive $\|A\|_2$ by definition of operator norm.
    \begin{align*}
        \begin{split}
            \|A\|_2^2 & = \max_x\frac{\|Ax\|_2^2}{\|x\|_2^2} = \max_x\frac{\|(\one_n^\top \otimes \I_m)x\|_2^2+\|(\I_n \otimes \one_m^\top)x\|_2^2}{\|x\|_2^2}\\
            & = \max_V\frac{\|(\one_n^\top \otimes \I_m)\vvec(V)\|_2^2+\|(\I_n \otimes \one_m^\top)\vvec(V)\|_2^2}{\|\vvec(V)\|_2^2}= \max_V\frac{\|\vvec( \I_mV\one_n)\|_2^2+\|\vvec(\one_m^\top V\I_n)\|_2^2}{\|\vvec(V)\|_2^2}\\
            & = \max_V\frac{\|V\one_n\|_2^2+\|V^\top\one_m\|_2^2}{\|V\|_F^2}\leq \max_V\frac{\|V\|_2^2\|\one_n\|_2^2+\|V^\top\|_2^2\|\one_m\|_2^2}{\|V\|_F^2} \leq n+m \ ,
        \end{split}
    \end{align*}
    where the fourth equality uses the property of Kronecker product $(P\otimes Q)\vvec(V)=\vvec(QVP^\top)$ for any matrices $P,Q,V$ that make the calculation consistent, and moreover, the equality is attained by $x=\tfrac{1}{\sqrt{mn}}\one_{mn}$. Thus $\|A\|_2=\sqrt{m+n}$ and for any sub-matrix $\subA$ of $A$,
    \begin{equation*}
        \|\subA\|_2\leq \|A\|_2=\sqrt{m+n}\ .
    \end{equation*}

    (ii) Note that $A=\begin{bmatrix}
        \one_n^\top \otimes \I_m \\ \I_n \otimes \one_m^\top
    \end{bmatrix}\in \mathbb R^{(m+n)\times mn}$ is totally uni-modular, and thus by definition of totally uni-modular matrix, any sub-matrix $\subA$ of $A\in \mathbb R^{(m+n)\times mn}$ is also totally uni-modular~\cite{schrijver1998theory}. Therefore, we know that $\subA_{ij}^{-1}\in\{-1,0,1\}$~\cite{schrijver1998theory} which further implies $\|\subA^{-1}\|_F\leq m+n$. We conclude the proof by noticing that $\|\subA^{-1}\|_2\leq \|\subA^{-1}\|_F\leq m+n$.
\end{proof}

\begin{lem}\label{lem:sharp-local}
    Let $\subA$ be a sub-matrix of $A$. Consider the homogeneous linear inequality system
    \begin{equation}\label{eq:generic-local-system}
        \subA u=0,\ u_S\geq 0,\ \subA^\top v \leq 0 \ ,
    \end{equation}
    where $u_S$ is a sub-vector of $u$. Denote $\alpha_L$ the sharpness of \eqref{eq:generic-local-system} and then it holds that
    \begin{equation*}
        \alpha_L\geq \frac{1}{m+n} \ .
    \end{equation*}
\end{lem}
\begin{proof}
    We use the characterization of sharpness constant derived in \cite[Corollary 1]{hinder2023worst}: $\alpha_L\geq\frac{1}{\max_{G\in\mathcal G}\|G^{-1}\|_2}$ where $\mathcal G$ is the set of all non-singular sub-matrices of the matrix $\begin{bmatrix}
        \subA^\top & 0 \\ 0 & \subA
    \end{bmatrix}$. From part (ii) of Proposition \ref{prop:norm-sharp}, we know $\|G^{-1}\|_2\leq m+n$ for any $G\in \mathcal G$ and thus $\alpha_L\geq \frac{1}{m+n}$.
\end{proof}

Furthermore, it turns out that vectorizing the PDHG iterates in Algorithm \ref{alg:rpd} is equivalent to PDHG iterates on standard LP. More precisely, let $\{(X^k,p^k,q^k)\}$ be PDHG iterates for solving \eqref{eq:ot} with update
    \begin{align*}
        &X^{k+1}\gets \mathrm{Proj}_{\mathbb R_+^{m\times n}}(X^{k}-\eta(C-p^{k}\one_n^\top-\one_m {q^{k}}^\top))\\
        & p^{k+1}\gets p^{k}+\sigma(f-(2X^{k+1}-X^{k})\one_n)\\
        & q^{k+1}\gets q^{k}+\sigma(g-(2X^{k+1}-X^{k})^\top\one_m) \ .
    \end{align*}
On the other hand, consider PDHG on solving standard LP \eqref{eq:standard-lp} with \eqref{eq:ot-lp-data}. The algorithm has update rule
\begin{align*}
        & x^{k+1}\gets \mathrm{Proj}_{\mathbb R_+^{mn}}(x^k-\eta(c-A^\top y^k)) \\ 
        & y^{k+1}\gets y^k+\sigma(b-A(2x^{k+1}-x^k)) \ ,
    \end{align*}
It is easily observed that for any $k\geq 0$, we have
\begin{equation*}
     x^{k}=\vvec(X^{k}),\quad y^k=\begin{bmatrix}
            p^k \\ q^k
        \end{bmatrix} \ .
\end{equation*}

Given this equivalence in iterates, now we prove the main result Theorem \ref{thm:local-rate}.
\begin{proof}[Proof of Theorem \ref{thm:local-rate}]
(i) We prove the existence of $T$, namely, Algorithm \ref{alg:rpd} has finite-time identification. Denote $z^{t,k}=(\vvec(X^{t,k}),(p^{t,k},q^{t,k}))$ and $z^{*}=(\vvec(X^{*}),(p^{*},q^{*}))$. By part (i) of Lemma \ref{lem:basic}, there exists an integer $T_0>0$ such that $\|z^{t,0}-z^*\|_2\leq \gamma\delta$ for any $t\geq T_0$. Denote $\gamma=\frac{\eta}{1+\eta}$. Therefore, for any $k\geq 1$,
\begin{equation*}
    X^{t,k}_{ij}\geq X^{*}_{ij}-\|X^{t,k}-X^{*}\|_2\geq X^{*}_{ij}-\|z^{t,0}-z^*\|_2>\delta-\gamma\delta>0, \ \forall (i,j)\in B_1 \ ,
\end{equation*}
and
\begin{equation*}
    X^{t,k}_{ij}<\gamma\frac{\delta}{2},\quad C_{ij}-p_i^{t,k}-q_j^{t,k}>\delta-\gamma{\delta}=\pran{1-{\gamma}}\delta>0, \quad \forall (i,j)\in N \ .
\end{equation*}
This implies for any $(i,j)\in N$ and any $k\geq 1$,
\begin{equation*}
    X^{t,k}_{ij}-\eta\pran{C_{ij}-p_i^{t,k}-q_j^{t,k}}<\pran{\gamma-\eta\pran{1-{\gamma}}}\delta\leq 0\ ,
\end{equation*}
and thus $X^{t,k}_{ij}=0$ for any $k\geq 1$. Let $T=T_0+1$ and we conclude the proof.

(ii) Denote $x=\vvec(X)$, $y=(p,q)$ and $\mathcal Z_v^*=\{(\vvec(X^*),p^*,q^*)|(X^*,p^*,q^*)\in \mathcal Z^*\}$. Then the PDHG update after identification can be written as
    \begin{align*}
        & x_{B\cup N_2}^{t,k+1}\gets \mathrm{Proj}_{\mathbb R^{|B|}\times \mathbb R_+^{|N_2|}}(x_{B\cup N_2}^{t,k}-\eta(c_{B\cup N_2}-A_{B\cup N_2}^\top y^{t,k}))=\mathrm{Proj}_{\mathbb R^{|B|}\times \mathbb R_+^{|N_2|}}(x_{B\cup N_2}^{t,k}-\eta A_{B\cup N_2}^\top (y^*-y^{t,k})) \\ 
        & y^{t,k+1}\gets y^{t,k}+\sigma(b-A_{B\cup N_2}(2x_{B\cup N_2}^{t,k+1}-x_{B\cup N_2}^{t,k}))= y^{t,k}-\sigma A_{B\cup N_2}(2(x_{B\cup N_2}^{t,k+1}-x^*)-(x_{B\cup N_2}^{t,k}-x^*))\ ,
    \end{align*}
    Denote $u^{t,k}=(u_B^{t,k},u_{N_2}^{t,k})=x_{B\cup N_2}^{t,k}-x_{B\cup N_2}^*$ and $v^{t,k}=y^{t,k}-y^*$. Note that $x_{N_2}^*=0$ and we have
    \begin{align*}
        & u^{t,k+1}\gets \mathrm{Proj}_{\mathbb R^{|B|}\times \mathbb R_+^{|N_2|}}(u^{t,k}+\eta A_{B\cup N_2}^\top v^{t,k}) \\ 
        & v^{t,k+1}\gets v^{t,k}-\sigma A_{B\cup N_2}(2u^{t,k+1}-u^{t,k})\ .
    \end{align*}
    Thus it turns out that $(u^{t,k},v^{t,k})$ actually solves the homogeneous inequality system
    \begin{equation}\label{eq:local-system}
        A_{B\cup N_2}u=0,\ u_{N_2}\geq 0,\ A_{B\cup N_2}^\top v \leq 0 \ .
    \end{equation}
    Denote $\mathcal K$ the solution set to \eqref{eq:local-system} and $\alpha_L$ the sharpness constant of \eqref{eq:local-system}. Hence combining part (ii) of Lemma \ref{lem:basic} and Lemma \ref{lem:sharp-local}, we have
     \begin{equation*}
        \tau^t\leq \frac{16\|A_{B\cup N_2}\|_2}{\alpha_L \beta} \leq \frac{16}{\beta}(m+n)^{1.5} \ ,
    \end{equation*}
    and
    \begin{equation}\label{eq:uv-linear}
        \mathrm{dist}((u^{t,0},v^{t,0}),\mathcal K)\leq \beta^{t-T}\frac{1}{\alpha_L}\mathrm{KKT}(u^{T,0},v^{T,0})\leq \beta^{t-T}(m+n)\mathrm{KKT}(u^{T,0},v^{T,0}) \ .
    \end{equation}
    Therefore we achieve
    \begin{equation*}
    \begin{aligned}
        \mathrm{dist}((X^{t,0},p^{t,0},q^{t,0}),\mathcal Z^*)&=\mathrm{dist}((x^{t,0},y^{t,0}),\mathcal Z_v^*)\leq \mathrm{dist}((u^{t,0},v^{t,0}),\mathcal K)\\
        &\leq \beta^{t-T}(m+n)\mathrm{KKT}(u^{T,0},v^{T,0})\leq \beta^{t-T}(m+n)\mathrm{KKT}(x^{T,0},y^{T,0})\\
        &=\beta^{t-T}(m+n)\mathrm{KKT}(X^{T,0},p^{T,0},q^{T,0}) \ ,
    \end{aligned}
    \end{equation*}
    where the first inequality is from $z^*+\mathcal K\subseteq \mathcal Z_v^*$, the second one uses \eqref{eq:uv-linear} and the last inequality is due to $\mathrm{KKT}(u^{T,0},v^{T,0})=\left\|\begin{bmatrix}
            A_{B\cup N_2}u^{T,0}\\ [A_{B\cup N_2}^\top v^{T,0}]^+
        \end{bmatrix}\right\|_2=\left\|\begin{bmatrix}
            A_{B\cup N_2}x_{B\cup N_2}^{T,0}-b\\ [A_{B\cup N_2}^\top y^{T,0}-c_{B\cup N_2}]^+
        \end{bmatrix}\right\|_2\leq \mathrm{KKT}(x^{T,0},y^{T,0})$.
    
\end{proof}

\subsection{Global complexity with Rational Data}\label{sec:global-rate}
In the preceding section, we demonstrate the fast linear convergence of Algorithm \ref{alg:rpd} once the iterates arrive in a local regime. In this section, we present a bound on the number of iterations needed to reach the local regime when the data of the problems are rational number. Together with the result in the previous section, this demonstrates a global convergence rate.

We start this section by introducing the following assumption:
\begin{ass}\label{ass:integer}
We assume the data of the optimal transport problems, i.e., $C_{ij}$, $f_i$ and $g_j$ for $1\le i\le m$ and $1\le j\le n$, are all rational numbers.
\end{ass}

Assumption \ref{ass:integer} rules out the unusual values of the data. This is a mild condition, and it is always satisfied in practice when floating point number representation is utilized. Furthermore, we define the precision of the data as follows:

\begin{mydef}
    We define the precision of the data $\Delta$ as the minimal positive constant such that for any $1\leq i\leq m$ and $1\leq j\leq n$, $C_{ij}=k_{ij}\Delta$, $f_i=s_i\Delta$ and $g_j=t_j\Delta$, where $k_{ij}$, $s_i$ and $t_j$ are integers.
\end{mydef}

Apparently, the existence of the precision $\Delta$ is guaranteed by Assumption \ref{ass:integer}. With the above definitions, our main theorem on global convergence of Algorithm \ref{alg:rpd} goes as follows
\begin{thm}\label{thm:global-rate}
    Consider $\{(X^{t,k},p^{t,k},q^{t,k})\}$ the iterates of Algorithm \ref{alg:rpd} with restart scheme \eqref{eq:restart-condition} for solving \eqref{eq:ot}. Suppose step-size $\eta\leq \frac{1}{2\sqrt{m+n}}$. Denote $z^*=(X^*,p^*,q^*)$ the converging optimal solution, namely $(X^{t,0},p^{t,0},q^{t,0})\rightarrow (X^*,p^*,q^*)$ and denote $(N, B_1, B_2)$ the partition using Definition \ref{def:partition}. Then it holds that
    \begin{equation*}
        T\leq \log_{1/\beta}\pran{\frac{96(1+\eta)H(m+n)^{2.5}\mathrm{KKT}(X^{0,0},p^{0,0},q^{0,0})}{\eta\delta}}+1 \ ,
    \end{equation*}
    and for any $t<T$, we have
        \begin{equation*}
        \tau^t\leq \frac{1536}{\beta}\frac{H}{\Delta}(m+n)^{3} \ ,
        \end{equation*}
        where $H=\|(\vvec(C),f,g)\|_\infty$.
\end{thm}
\begin{rem}
    Theorem \ref{thm:global-rate} provides a upper bound on identification complexity $T$ in Theorem \ref{thm:local-rate}. It turns out the dependence of the rate on near-degeneracy metric is merely logarithmic, achieving an $\widetilde O(n^{3.5})$ iteration complexity.
\end{rem}
\begin{rem}
    We here also comment that the length of the inner loop $\tau^t$ for $t\le T$ is inverse in the precision of the data $\Delta$ -- the finer grid of the data, the longer it may take for each inner loop.
\end{rem}

Combining Theorem \ref{thm:local-rate} and Theorem \ref{thm:global-rate}, we can obtain the global complexity result of Algorithm \ref{alg:rpd} under Assumption \ref{ass:integer}:

\begin{cor}
Combining Theorem \ref{thm:local-rate} and Theorem \ref{thm:global-rate} implies $\widetilde O\pran{(m+n)^{3.5}\Delta+(m+n)^{1.5}\log\frac{1}{\epsilon}}$ global iteration complexity to find an $\epsilon$-accuracy solution. Since each iteration of Algorithm \ref{alg:rpd} requires $O(mn)$ operations, the global flop count equals $\widetilde O\pran{mn(m+n)^{3.5}\Delta+mn(m+n)^{1.5}\log\frac{1}{\epsilon}}$. Moreover, it is straightforward to implement Algorithm \ref{alg:rpd} in parallel to achieve $\widetilde O\pran{(m+n)^{3.5}\Delta+(m+n)^{1.5}\log\frac{1}{\epsilon}}$ global parallel depth. The same iteration complexity, flop counts and parallel depth hold to find a feasible $\epsilon$-accuracy solution by rounding (Algorithm \ref{alg:round}).
\end{cor}

The rest of this section proves Theorem \ref{thm:global-rate}. We begin with a connection with a scaled version of PDHG iterates. Consider the following auxiliary LP
\begin{equation}\label{eq:tu-lp}
        \min_{\tilde x\geq 0,A\tilde x=\tilde b}\ \tilde c^\top\tilde x \ ,
\end{equation}
where  $\tilde c=\frac{1}{\Delta}\vvec(C)\in \mathbb Z^{mn}$, $\tilde b=\frac{1}{\Delta}\begin{bmatrix}
    f \\ g
\end{bmatrix}\in \mathbb Z^{m+n}$ and  $A=\begin{bmatrix}
        \one_n^\top \otimes \I_m \\ \I_n \otimes \one_m^\top
    \end{bmatrix}\in \mathbb Z^{(m+n)\times mn}$. 
PDHG on solving \eqref{eq:tu-lp} has update rules 
\begin{align*}
        & \tilde x^{k+1}\gets \mathrm{Proj}_{\mathbb R_+^{mn}}(\tilde x^k-\eta(\tilde c-A^\top \tilde y^k)) \\ 
        & \tilde y^{k+1}\gets \tilde y^k+\sigma(\tilde b-A(2\tilde x^{k+1}-\tilde x^k)) \ .
    \end{align*}
Recall the PDHG iterates for solving \eqref{eq:ot}
\begin{align*}
        &X^{k+1}\gets \mathrm{Proj}_{\mathbb R_+^{m\times n}}(X^{k}-\eta(C-p^{k}\one_n^\top-\one_m {q^{k}}^\top))\\
        & p^{k+1}\gets p^{k}+\sigma(f-(2X^{k+1}-X^{k})\one_n)\\
        & q^{k+1}\gets q^{k}+\sigma(g-(2X^{k+1}-X^{k})^\top\one_m) \ .
    \end{align*}
We know that for any $k\geq 0$, $\tilde x^{k}=\frac{1}{\Delta}\vvec(X^{k})$ and $\tilde y^k=\frac{1}{\Delta}\begin{bmatrix}
            p^k \\ q^k
        \end{bmatrix}$.

Note that iterates $(\tilde x^k,\tilde y^k)$ actually solve LP stated in \eqref{eq:tu-lp}. The lemma below establishes the sharpness of LP \eqref{eq:tu-lp}.
\begin{lem}\label{lem:sharp-global}
Denote $\alpha$ the sharpness constant of KKT system of \eqref{eq:tu-lp}. Then it holds that
    \begin{equation*}
        {\alpha}\geq \frac{1}{\tfrac{96H}{\Delta}(m+n)^{2.5}} \ ,
    \end{equation*}
    where $H=\|(\vvec(C),f,g)\|_\infty$.
\end{lem}
\begin{proof}
    Notice that LP \eqref{eq:tu-lp} is totally uni-modular, that is, constraint matrix is totally uni-modular, and right-hand-side and objective vector are integral. By {\cite[Lemma 5]{hinder2023worst}}, the sharpness constant of \eqref{eq:tu-lp} can be lower bounded as
    \begin{equation*}
        \alpha \geq \frac{1}{4(m+n)+2+\tfrac{2H}{\Delta}(2(m+n)+1)^{2.5}} \ ,
    \end{equation*}
    Note that $\tfrac{H}{\Delta}\geq 1$ and $2(m+n)+1\leq 4(m+n)$ and we have
    \begin{equation*}
       4(m+n)+2+\tfrac{2H}{\Delta}(2(m+n)+1)^{2.5}\leq (1+\tfrac{2H}{\Delta})(2(m+n)+1)^{2.5}\leq \tfrac{96H}{\Delta}(m+n)^{2.5} \ ,
    \end{equation*}
    and we have ${\alpha}\geq \frac{1}{\tfrac{96H}{\Delta}(m+n)^{2.5}}$.
\end{proof}
The following lemma provides global bounds on restart frequency and builds linear convergence of iterates.
\begin{lem}\label{lem:global-rate}
    Consider $\{(X^{t,k},p^{t,k},q^{t,k})\}$ the iterates of Algorithm \ref{alg:rpd} with restart scheme \eqref{eq:restart-condition} for solving \eqref{eq:ot} satisfying Assumption \ref{ass:integer}. Suppose the step-size $\eta\leq \frac{1}{2\sqrt{m+n}}$. Then it holds for any $t\geq 0$ that the restart frequency can be upper bounded as 
    \begin{equation*}
        \tau^t\leq \frac{1536}{\beta}\frac{H}{\Delta}(m+n)^{3} \ ,
        \end{equation*}
    and the linear convergence holds
    \begin{equation*}
        \mathrm{dist}((X^{t,0},p^{t,0},q^{t,0}),\mathcal Z^*)\leq 96\beta^t{H}(m+n)^{2.5}\mathrm{KKT}(X^{0,0},p^{0,0},q^{0,0}) \ ,
    \end{equation*}
    where $H=\|(\vvec(C),f,g)\|_\infty$.
\end{lem}
\begin{proof}
Denote $\widetilde{\mathrm{KKT}}(\tilde x,\tilde y)=\left\|\begin{bmatrix}
        A\tilde x-\tilde b \\ 
        [A^\top \tilde y-\tilde c]^+ \\ \frac{\Delta}{R}(\tilde c^\top \tilde x-\tilde b^\top \tilde y)
    \end{bmatrix}\right\|_2$ and note that
\begin{equation*}
    \mathrm{KKT}(X^{t,k},p^{t,k},q^{t,k})=\widetilde{\mathrm{KKT}}(\tilde x^{t,k},\tilde y^{t,k}){\Delta} \ .
\end{equation*}
By part (ii) of Lemma \ref{lem:basic}, we have 
    \begin{equation*}
        \tau^t\leq \frac{16\|A\|_2}{\beta\alpha}\leq \frac{16(m+n)^{0.5}}{\beta \frac{1}{\tfrac{96H}{\Delta}(m+n)^{2.5}}}=\frac{1536}{\beta}\frac{H}{\Delta}(m+n)^3 \ .
    \end{equation*}
Furthermore, it holds that
    \begin{equation*}
    \begin{aligned}
        & \mathrm{dist}((X^{t,0},p^{t,0},q^{t,0}),\mathcal Z^*)=\Delta \mathrm{dist}((\tilde x^{t,0},\tilde y^{t,0}),\tfrac{1}{\Delta}\mathcal Z_v^*)\\
        & \leq \Delta\beta^{t}\frac{1}{\alpha}\widetilde{\mathrm{KKT}}(\tilde x^{t,0},\tilde y^{t,0})\leq96\beta^t{H}(m+n)^{2.5}\mathrm{KKT}(X^{0,0},p^{0,0},q^{0,0}) \ ,
    \end{aligned}
    \end{equation*}
    where the first equality is due to definition of $(\tilde x,\tilde y)$, the last equality is from Lemma \ref{lem:sharp-global}, and the inequality uses part (iii) of Lemma \ref{lem:basic}.
\end{proof}

We prove Theorem \ref{thm:global-rate} below that follows from Theorem \ref{thm:local-rate} and Lemma \ref{lem:global-rate}.
\begin{proof}[Proof of Theorem \ref{thm:global-rate}]
    The proof of Theorem \ref{thm:local-rate} implies that the condition $\|z^{t,0}-z^*\|_2\leq \frac{\eta}{1+\eta}\delta$ is sufficient to achieve identification (part (i) of Theorem \ref{thm:local-rate}). Therefore, in light of Lemma \ref{lem:global-rate},
    \begin{equation*}
        \mathrm{dist}((X^{t,0},p^{t,0},q^{t,0}),\mathcal Z^*) \leq 96\beta^tH(m+n)^{2.5}\mathrm{KKT}(X^{0,0},p^{0,0},q^{0,0})\leq \frac{\eta}{1+\eta}\delta  \ ,
    \end{equation*}
    which is equivalent to
    \begin{equation*}
        t\geq \log_{1/\beta}\pran{\frac{96(1+\eta)H(m+n)^{2.5}\mathrm{KKT}(X^{0,0},p^{0,0},q^{0,0})}{\eta\delta}} \ .
    \end{equation*}
    Thus $T\leq\log_{1/\beta}\pran{\frac{96(1+\eta)H(m+n)^{2.5}\mathrm{KKT}(X^{0,0},p^{0,0},q^{0,0})}{\eta\delta}}+1$ and the upper bound of restart length $\tau^t$ is proven in Lemma \ref{lem:global-rate}.
\end{proof}

\section{PDOT.jl: A Julia implemented solver for OT}\label{sec:numerical}
In this section, we present PDOT.jl, a OT solver based on Algorithm \ref{alg:rpd}. The solver is available at \href{https://github.com/jinwen-yang/PDOT.jl}{https://github.com/jinwen-yang/PDOT.jl}. We compare the numerical performance of PDOT with Sinkhorn and barrier method of Gurobi. 

PDOT leverages several algorithmic enhancements upon Algorithm \ref{alg:rpd} to improve its practical performance. More specifically,
\begin{itemize}
    \item {\bf Adaptive restart.} PDOT adopts similar adaptive restarting strategy of cuPDLP~\cite{lu2023cupdlp}. The restart scheme is described as follows: Define the restart candidate as 
    {\small
    \begin{equation*}
        \begin{aligned}
            (X_c^{t,k+1},p_c^{t,k+1},q_c^{t,k+1})&:=\mathrm{GetRestartCandidate}((X^{t,k+1},p^{t,k+1},q^{t,k+1}),(\bar X^{t,k+1},\bar p^{t,k+1},\bar q^{t,k+1}))\\
            &=\begin{cases}
            (X^{t,k+1},p^{t,k+1},q^{t,k+1}),& \mathrm{KKT}(X^{t,k+1},p^{t,k+1},q^{t,k+1})<\mathrm{KKT}(\bar X^{t,k+1},\bar p^{t,k+1},\bar q^{t,k+1}) \\ (\bar X^{t,k+1},\bar p^{t,k+1},\bar q^{t,k+1}),& \mathrm{otherwise} \ .
        \end{cases}
        \end{aligned}
    \end{equation*} 
    }
The algorithm restarts if one of three conditions holds:
\begin{enumerate}
    \item[(i)] (Sufficient decay in relative KKT error)
    \begin{equation*}
        \mathrm{KKT}(X_c^{t,k+1},p_c^{t,k+1},q_c^{t,k+1})\leq \beta_{\mathrm{sufficient}}\mathrm{KKT}(X_c^{t,0},p_c^{t,0},q_c^{t,0}) \ ,
    \end{equation*}
    \item[(ii)] (Necessary decay + no local progress in relative KKT error)
    \begin{equation*}
        \mathrm{KKT}(X_c^{t,k+1},p_c^{t,k+1},q_c^{t,k+1})\leq \beta_{\mathrm{necessary}}\mathrm{KKT}(X_c^{t,0},p_c^{t,0},q_c^{t,0}) \ ,
    \end{equation*}
    and
    \begin{equation*}
        \mathrm{KKT}(X_c^{t,k+1},p_c^{t,k+1},q_c^{t,k+1})>\mathrm{KKT}(X_c^{t,k},p_c^{t,k},q_c^{t,k}) \ ,
    \end{equation*}
    \item[(iii)] (Long inner loop)
    \begin{equation*}
        k\geq \beta_{\mathrm{artificial}}\times\texttt{total\_iteration} \ ,
    \end{equation*}
\end{enumerate}
where parameters $\beta_{\mathrm{sufficient}}=0.1$, $\beta_{\mathrm{necessary}}=0.9$ and $\beta_{\mathrm{artificial}}=0.36$.

    \item {\bf Adaptive step-size and primal weight.} The primal and the dual step-sizes of PDOT are reparameterized as
    \begin{equation*}
        \tau = \eta/\omega,\; \sigma=\eta\omega\quad \text{with}\;\  \eta,\omega>0\ ,
    \end{equation*}
    where $\eta$ (called step-size) controls the scale of the step-sizes, and $\omega$ (called primal weight) balances the primal and the dual progress. The step-size $\eta^{t,k}$ at inner iteration $k$ is selected as 
    find a step-size by a heuristic line search that satisfies
    \begin{equation*}\label{eq:step-size}
            \eta^{t,k}\leq \frac{\sqrt{\omega\left\|\vvec(X^{t,k+1}-X^{t,k})\right\|_2^2+\frac{1}{\omega}\left\|(p^{t,k+1}-p^{t,k},q^{t,k+1}-q^{t,k})\right\|_2^2}}{2\pran{(p^{t,k+1}-p^{t,k})^\top (X^{t,k+1}-X^{t,k})\one_n+(q^{t,k+1}-q^{t,k})^\top(X^{t,k+1}-X^{t,k})^\top\one_m}}\ ,
        \end{equation*}
    where $w$ is the current primal weight. See \cite{applegate2021practical} for more detail.
    
    The update of primal weight is specific during restart occurrences, thus infrequently. More precisely, the initial of $\omega=1$. Let $\Delta_X^t=\left\|\vvec(X^{t,0}-X^{t-1,0})\right\|_2$ and $\Delta_{(p,q)}^t=\left\|(p^{t,0}-p^{t-1,0},q^{t,0}-q^{t-1,0})\right\|_2$. PDOT initiates the primal weight update at the beginning of each new epoch with $\theta=0.5$.
    {\small
    \begin{equation*}
    \mathrm{PrimalWeightUpdate}(\Delta_X^t,z^{t-1,0},\Delta_{(p,q)}^t, \omega^{t-1}):=\begin{cases}
        \exp\pran{\theta \log\pran{\frac{\Delta_{(p,q)}^t}{\Delta_X^t}}+(1-\theta)\log\omega^{t-1}},\; & \Delta_X^t,\Delta_{(p,q)}^t>\epsilon_{\mathrm{zero}}\\
        \omega^{t-1}, \; & \mathrm{otherwise}
    \end{cases} \ .
    \end{equation*}
    }
    The adaptive step-size rule and primal weight update are extensions of those in PDLP \cite{applegate2021practical}.
\end{itemize}

\subsection{Experimental setup}
{\bf Datasets.} We conducted experiments using various algorithms on the DOTmark dataset, a discrete optimal transport benchmark~\cite{schrieber2016dotmark}. The dataset consists of ten distinct classes of images, with each class containing ten images. These images are derived from multiple sources: simulations based on different probability distributions (\texttt{CauchyDensity}, \texttt{GRFmoderate}, \texttt{GRFrough}, \texttt{GRFsmooth}, \texttt{LogGRF}, \texttt{LogitGRF}, \texttt{WhiteNoise}), geometric configurations (\texttt{Shapes}), classical test images (\texttt{ClassicImages}), and microscopic scientific observations (\texttt{Microscopy}). Within each class, the optimal transport problem can be formulated between any pair of images, resulting in 45 distinct problems per class and a total of 450 problems across the dataset.

\begin{figure}[ht!]
	\centering
     \hspace{-0.5cm}
	\begin{tabular}{c c c c}
		\includegraphics[width=\textwidth]{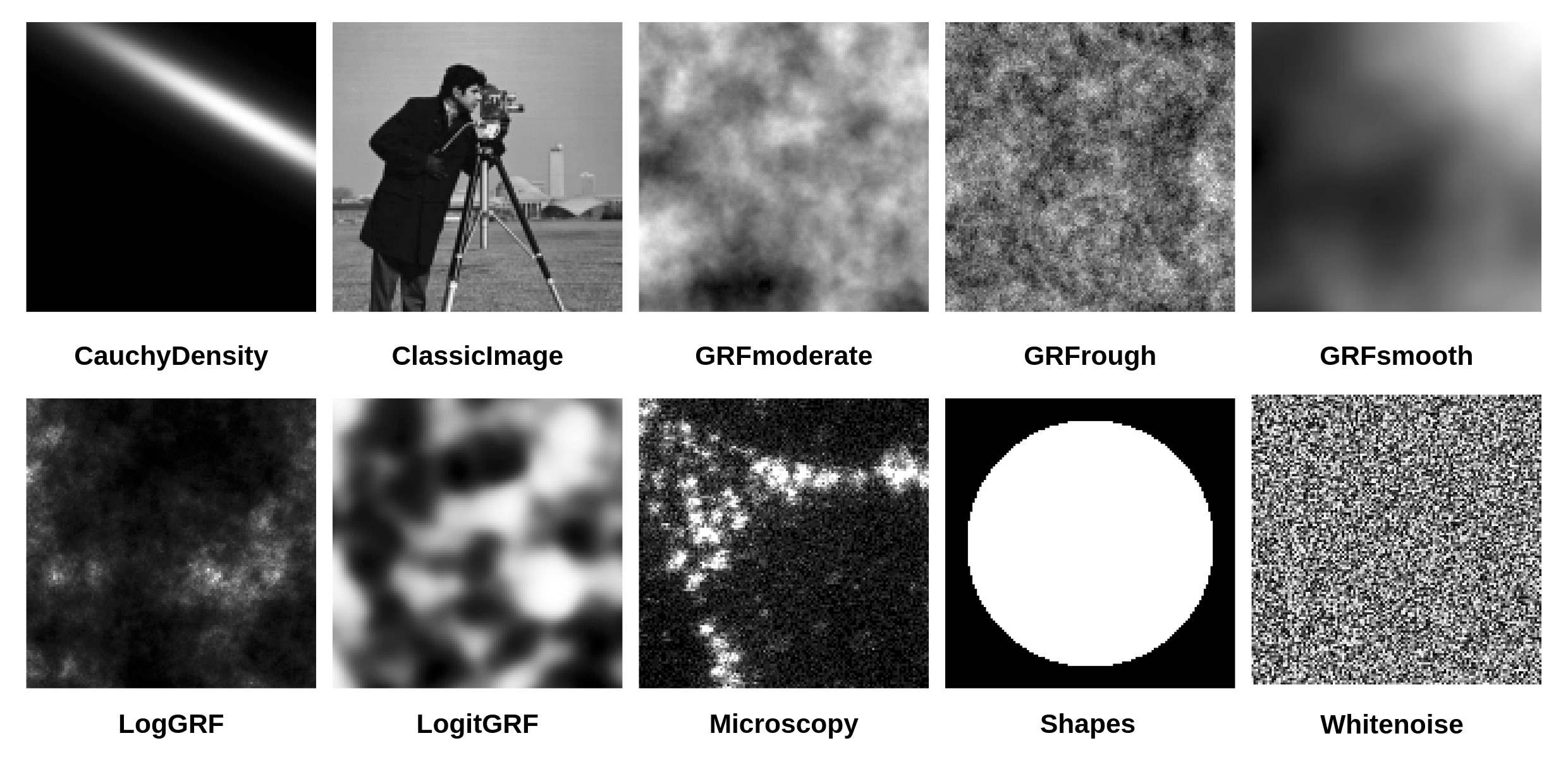}
	\end{tabular}
	\caption{Example images of each class from DOTmark set~\cite{schrieber2016dotmark}.}
	\label{fig:data}
\end{figure}

The images are processed at various resolutions, ranging from $32 \times 32$ to $128\times 128$ in our experiments. Statistics of the problems considered are summarized in Table \ref{tab:stats}.

\begin{table}[ht]
\centering
{\small
\begin{tabular}{|c|ccc|}
\hline
\textbf{Resolution}               & $32\times32$        & $64\times 64$        & $128\times128$         \\ \hline
\textbf{Dimension of cost matrix} & $1024\times1024$ & $4096\times4096$ & $16384\times16384$ \\\hline
\textbf{Size of cost matrix}       & 8.389MB   & 134.218MB & 2.147GB     \\\hline
\textbf{Number of variables}      & 1048576   & 16777216  & 268435456   \\\hline
\textbf{Number of constraints}    & 2048      & 8192      & 32768       \\ \hline
\end{tabular}}
\caption{Statistics of instances considered.}
\label{tab:stats}
\end{table}

As cost function, we consider the $\ell_1$, $\ell_2$ and $\ell_\infty$ distances as in \cite{zanetti2023interior}. Denote coordinate tuples $(\alpha_i,\beta_i)$ and $(\alpha_j,\beta_j)$ satisfying $A_{\alpha_i,\beta_i}=\vvec(A)_i$ and $B_{\alpha_j,\beta_j}=\vvec(B)_j$, respectively. Denote $C_{ij}^k$ the cost of moving mass from position $i$ to position $j$ under $\ell_k$ norm, $k\in\{1,2,\infty\}$ and it is defined as follows
\begin{equation*}
    C_{ij}^1=|\alpha_i-\alpha_j|+|\beta_i-\beta_j|, \ C_{ij}^2=\sqrt{(\alpha_i-\alpha_j)^2+(\beta_i-\beta_j)^2}, \ C_{ij}^\infty=\max\{|\alpha_i-\alpha_j|,|\beta_i-\beta_j|\} \ .
\end{equation*}

In summary, we consider 4050 OT instances in total with 3 resolutions, while there are 10 different image classes for each resolution. Each class contains 10 pictures and are grouped into 45 pairs, namely 45 pairs of marginal distributions. The cost matrix for each resolution is constructed using 3 different norms.

{\bf Software.} PDOT.jl and Sinkhorn are implemented in an open-source Julia module and utilizes {CUDA.jl} as the interface for working with NVIDIA CUDA GPUs using Julia. Due to the well-known numerical issue of Sinkhorn, log-domain stabilized Sinkhorn \cite{schmitzer2019stabilized} is implemented to improve the numerical stability in this paper. The penalty parameter of Sinkhorn is set to 0.01/0.001/0.0001 and the results under each parameter is reported separately. Barrier method (with crossover disabled) and network simplex implemented in Gurobi are compared as well. The running time of PDOT.jl and Sinkhorn is measured after pre-compilation in Julia. 

{\bf Rounding.} Transportation plans reported by all solvers are rounded using Algorithm \ref{alg:round} to obtain feasible solutions, and the reported results are based on the rounded feasible solutions.

{\bf Computing environment.} We use NVIDIA H100-PCIe-80GB GPU, with CUDA 12.3, for running PDOT and Sinkhorn, and we use Intel Xeon Gold 6248R CPU 3.00GHz with 128GB RAM and 8 threads for running Gurobi. The experiments are performed in Julia 1.9.2 and Gurobi 11.0.

{\bf Termination criteria.} PDOT terminates when the relative KKT error is no greater than the termination tolerance $\epsilon=10^{-4}$. Furthermore, we terminate Sinkhorn when the primal feasibility is smaller than or equal to tolerance $\epsilon=10^{-4}$. This is same as Sinkhorn implemented in existing OT solvers such as \cite{flamary2021pot}. We also set $10^{-4}$ tolerances for parameters \texttt{FeasibilityTol}, \texttt{OptimalityTol} and \texttt{BarConvTol} for Gurobi barrier method. We use the default termination for Gurobi network simplex method.

{\bf Time limit.} We impose a time limit of 3600 seconds on all instances.

{\bf Shifted geometric mean.} We report the shifted geometric mean of solve time to measure the performance of solvers on a certain collection of problems. More precisely, shifted geometric mean is defined as $\left(\prod_{i=1}^n (t_i+\mu)\right)^{1/n}-\mu$ where $t_i$ is the solve time for the $i$-th instance. We shift by $\mu=10$ and denote it SGM10. If the instance is unsolved, the solve time is always set to the corresponding time limit. 

{\bf Optimality gap.} We measure optimality using duality gap. More precisely, given primal-dual pair $(X,p,q)$, the duality gap is defined as 
\begin{equation*}
    \mathrm{Gap}(X,p,q)=\left|\langle C,X\rangle- f^\top p-g^\top q\right| \ .
\end{equation*}
We summarize the duality gap across different instances using (non-shifted) geometric mean.

\subsection{Results on benchmark dataset}
\begin{figure}[ht!]
    \hspace{-1cm}
	\begin{tabular}{c c c c}
		& \includegraphics[width=0.33\textwidth]{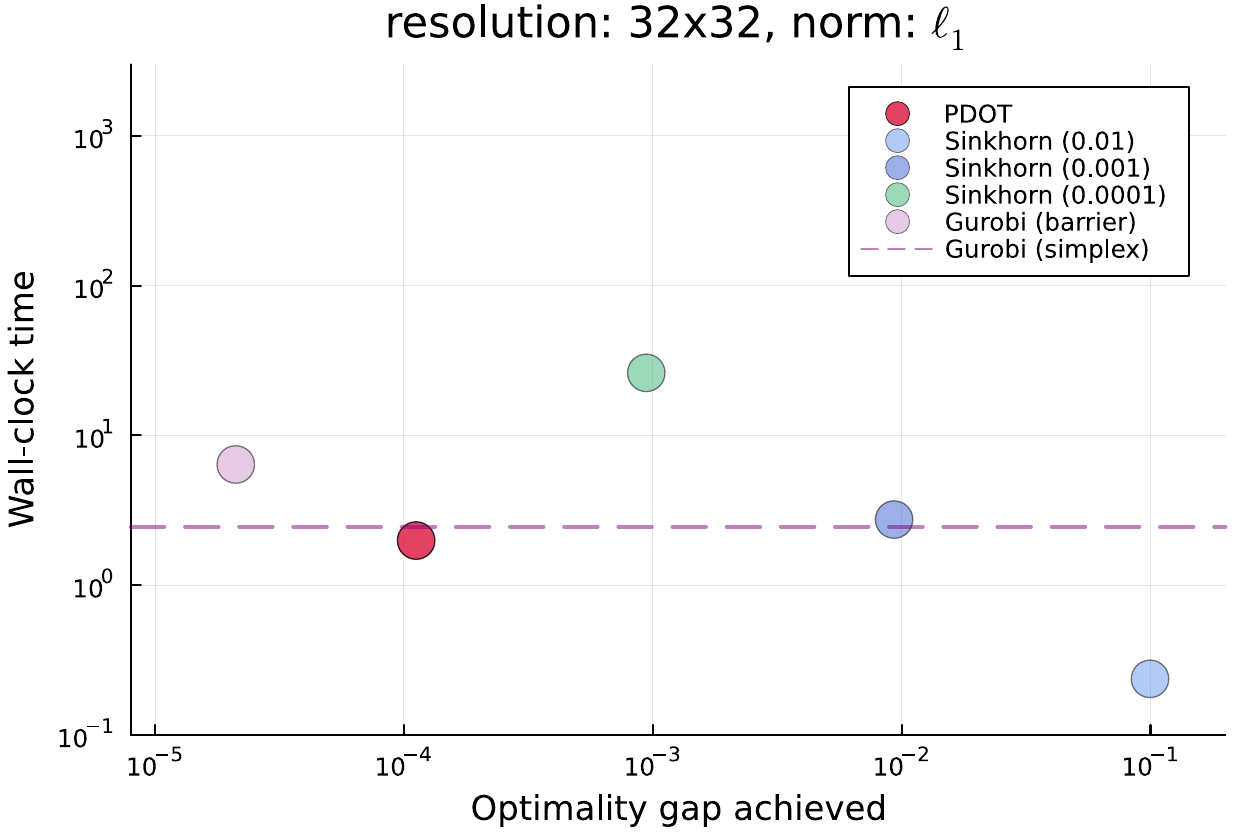}
        & \includegraphics[width=0.33\textwidth]{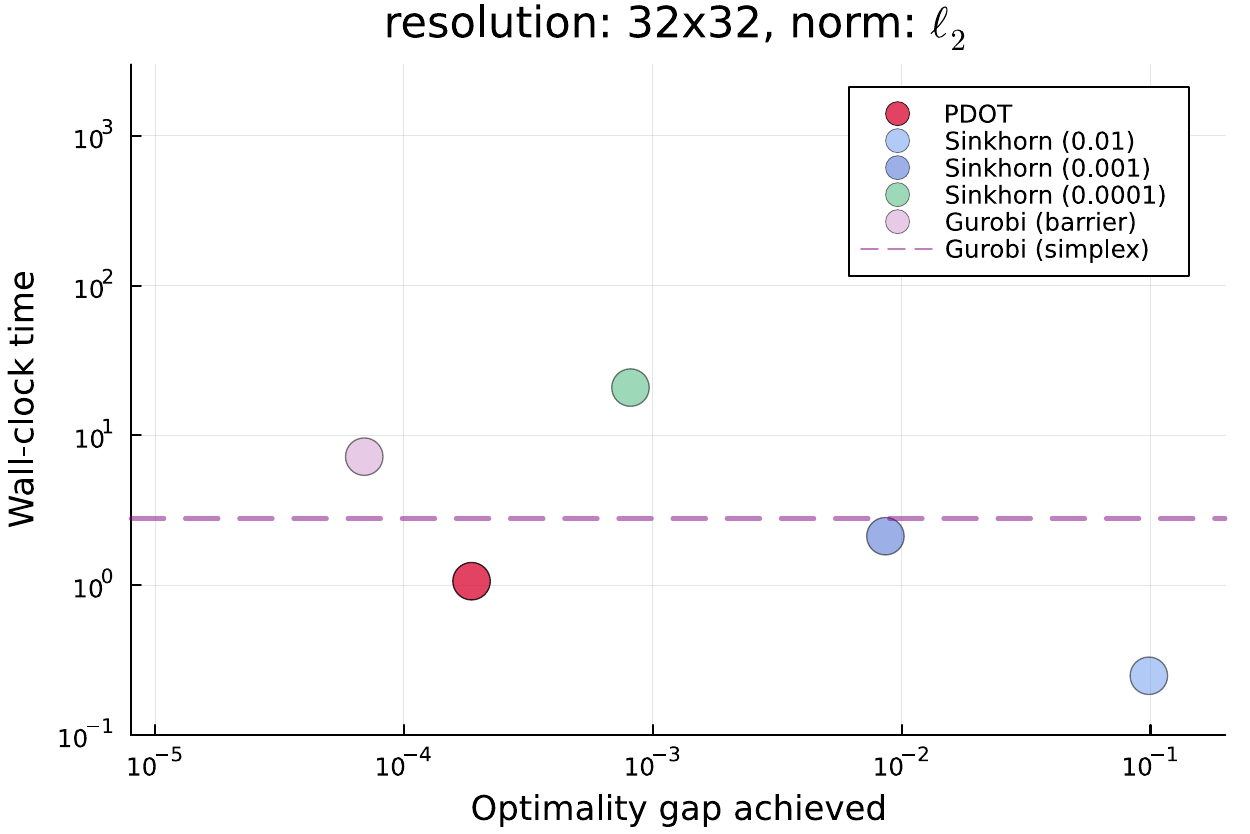}
        & \includegraphics[width=0.33\textwidth]{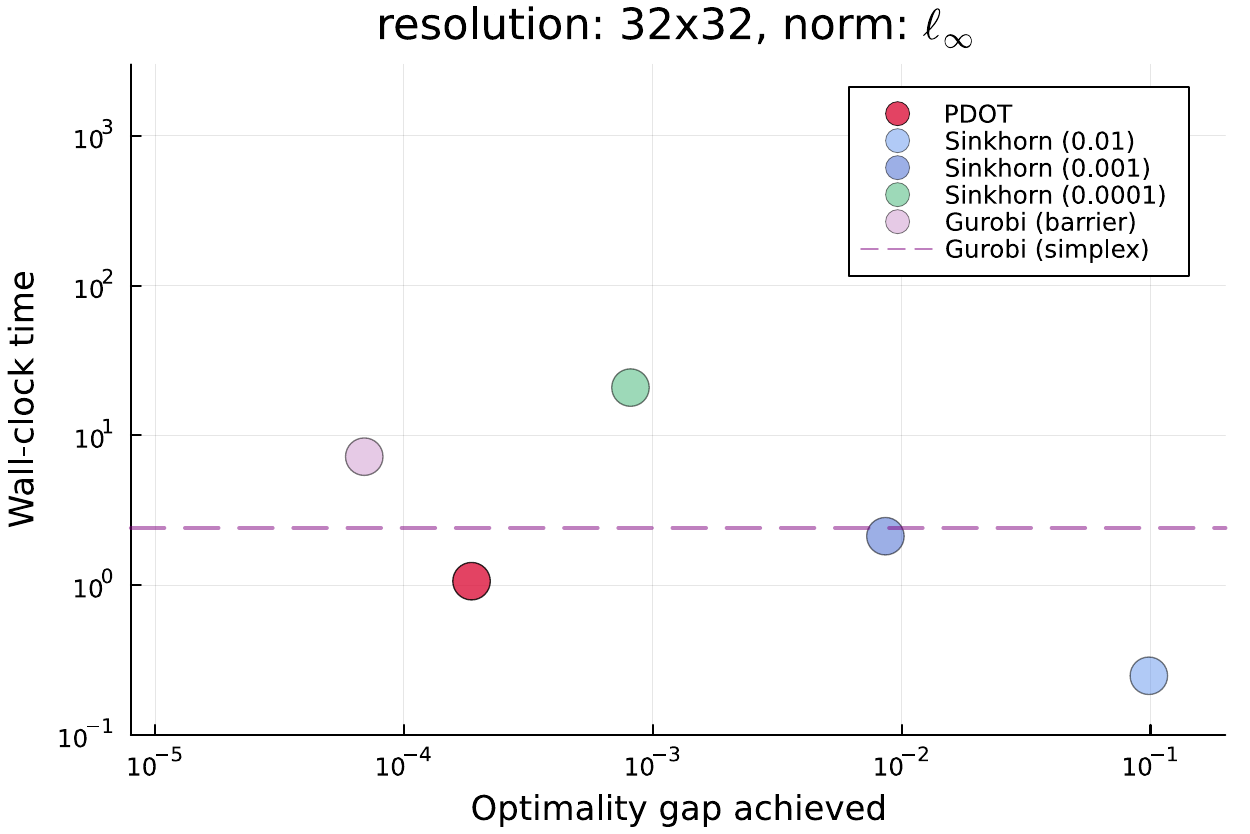}\\
        & \includegraphics[width=0.33\textwidth]{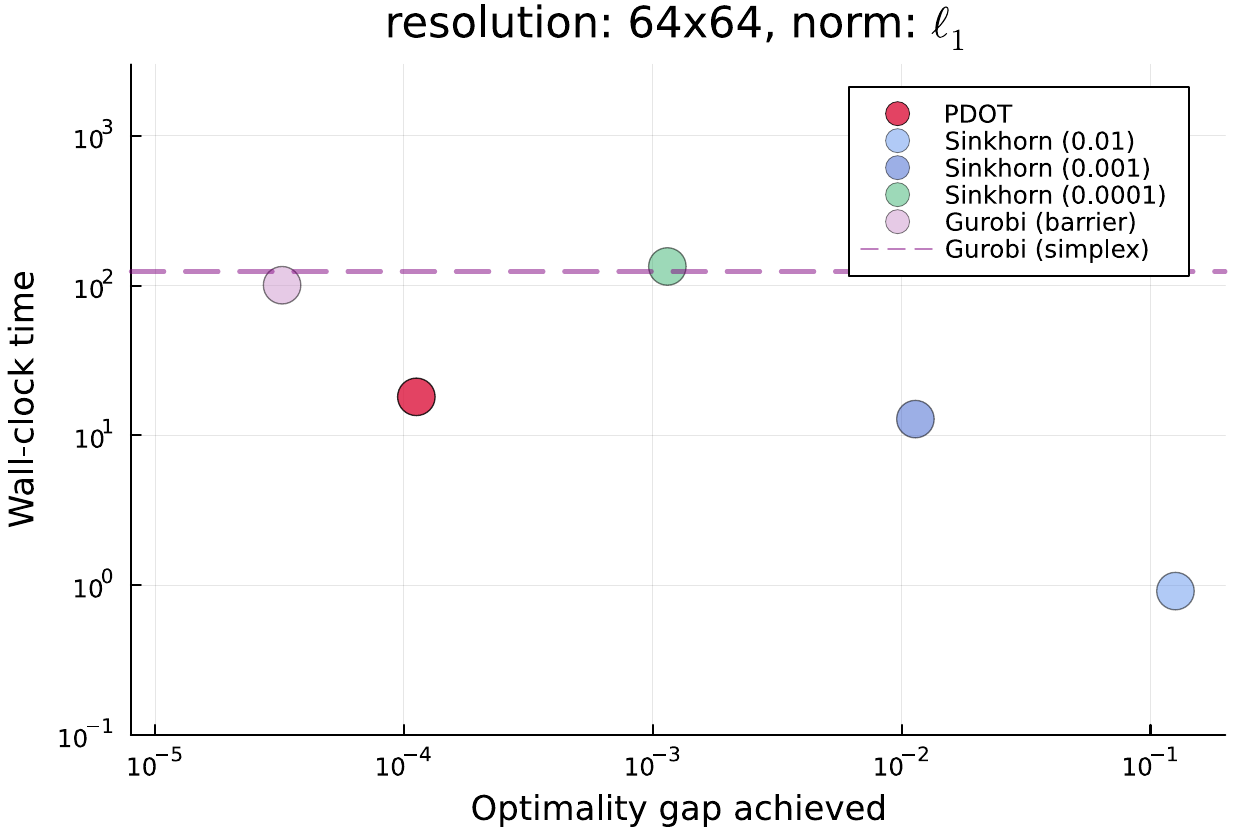}
        & \includegraphics[width=0.33\textwidth]{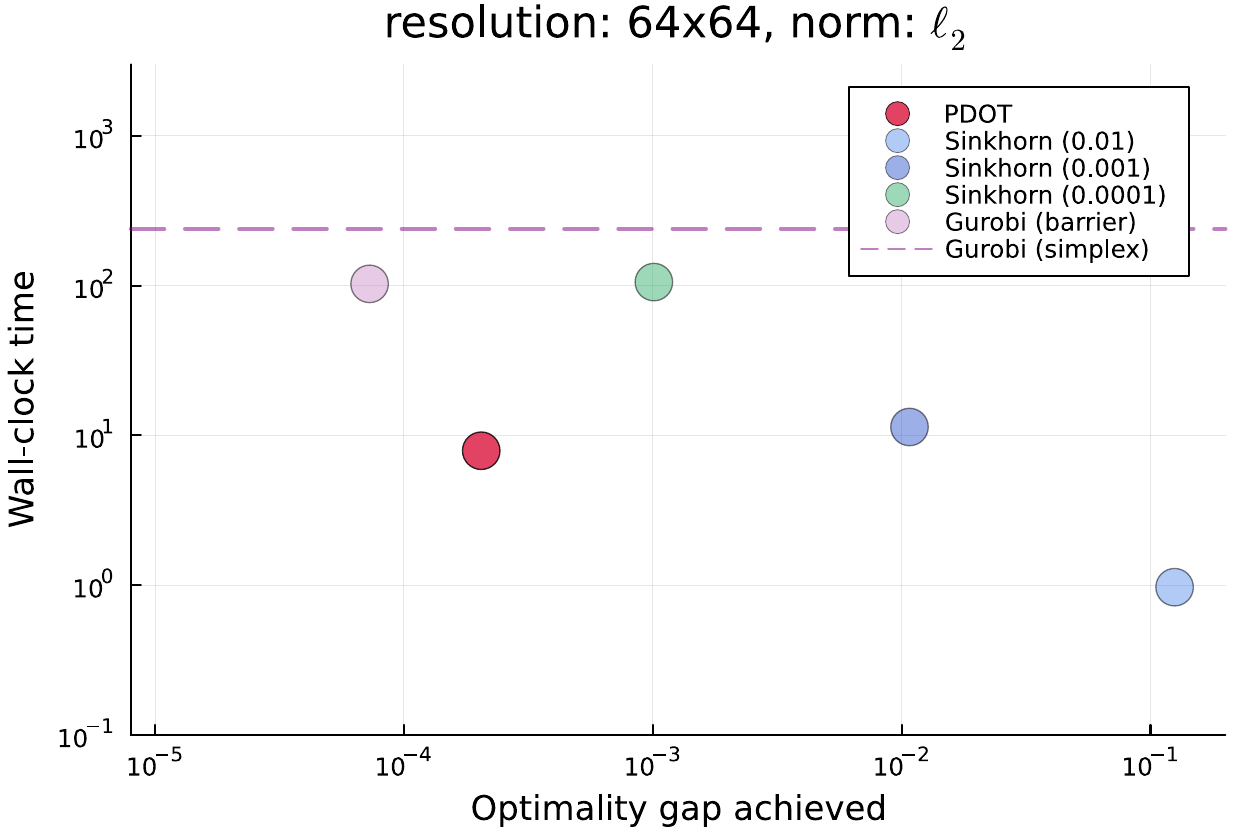}
        & \includegraphics[width=0.33\textwidth]{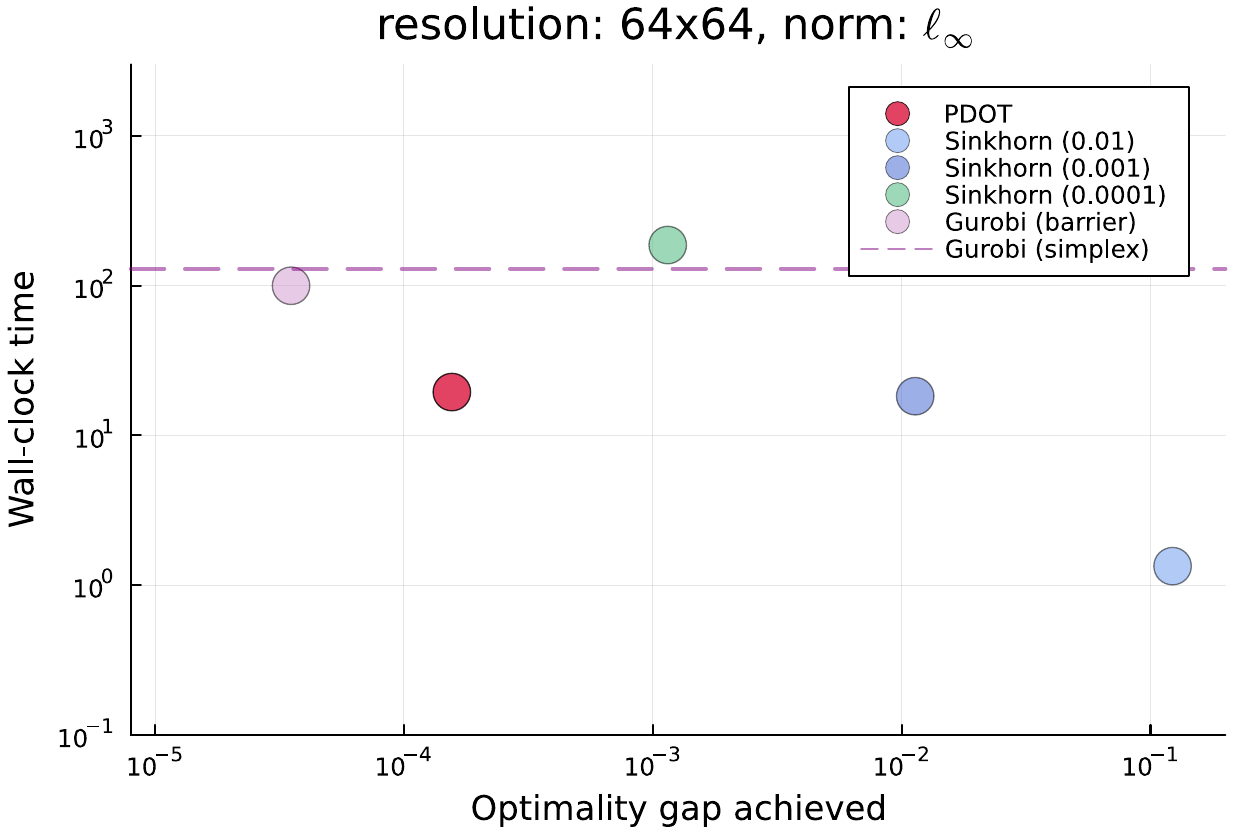}\\
        & \includegraphics[width=0.33\textwidth]{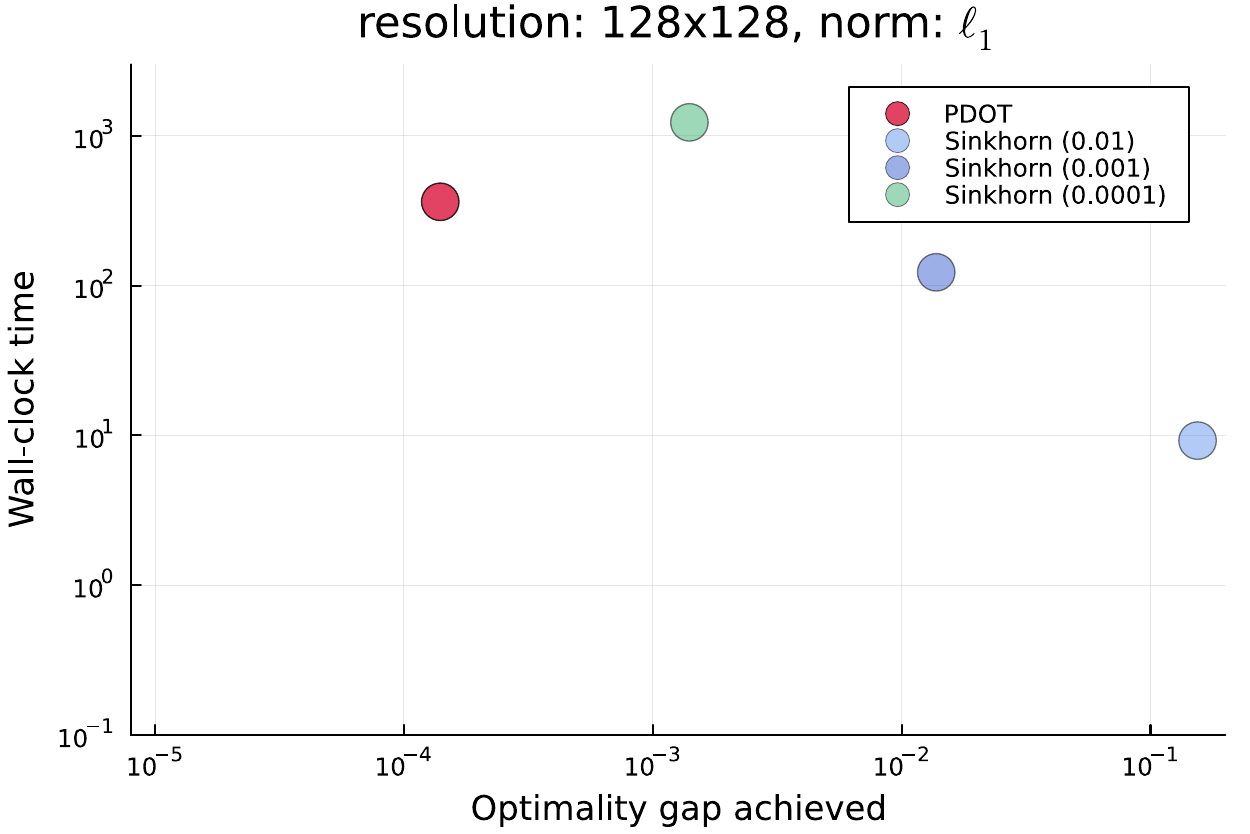}
        & \includegraphics[width=0.33\textwidth]{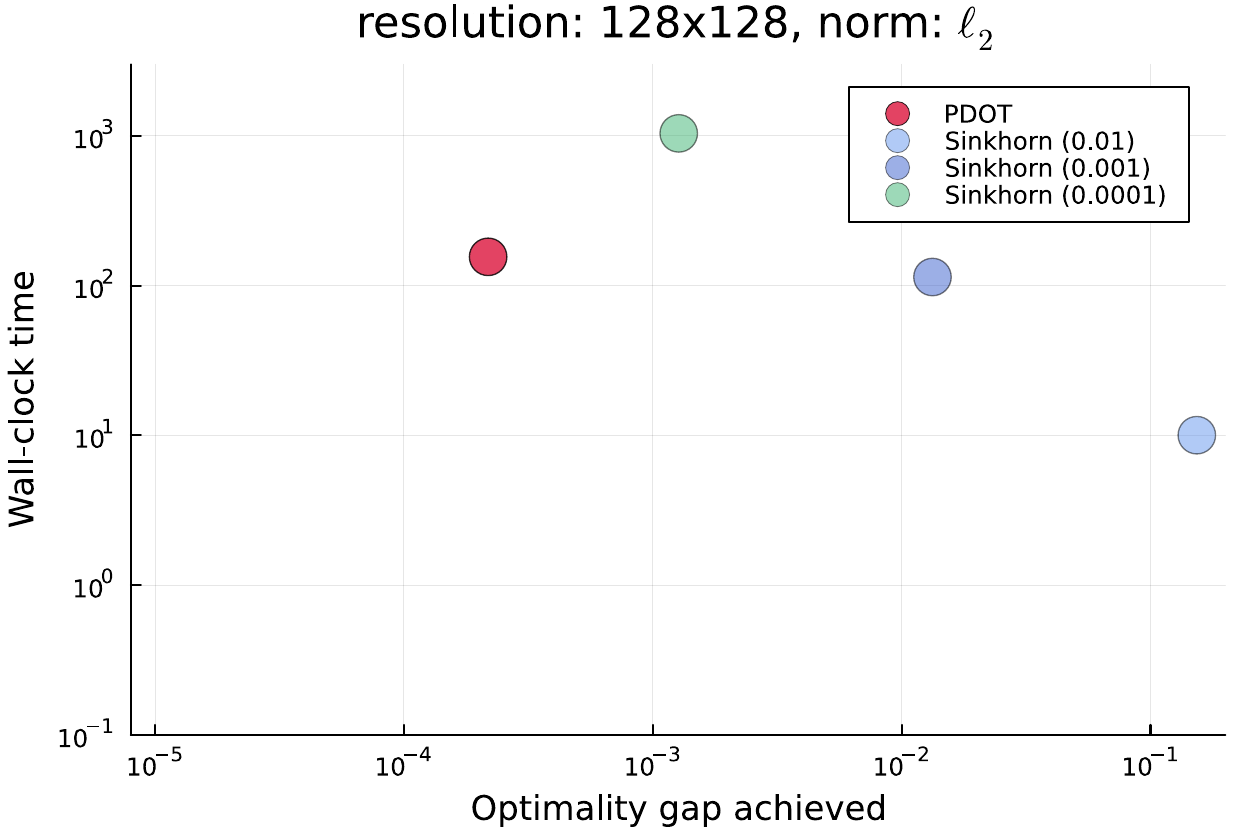}
        & \includegraphics[width=0.33\textwidth]{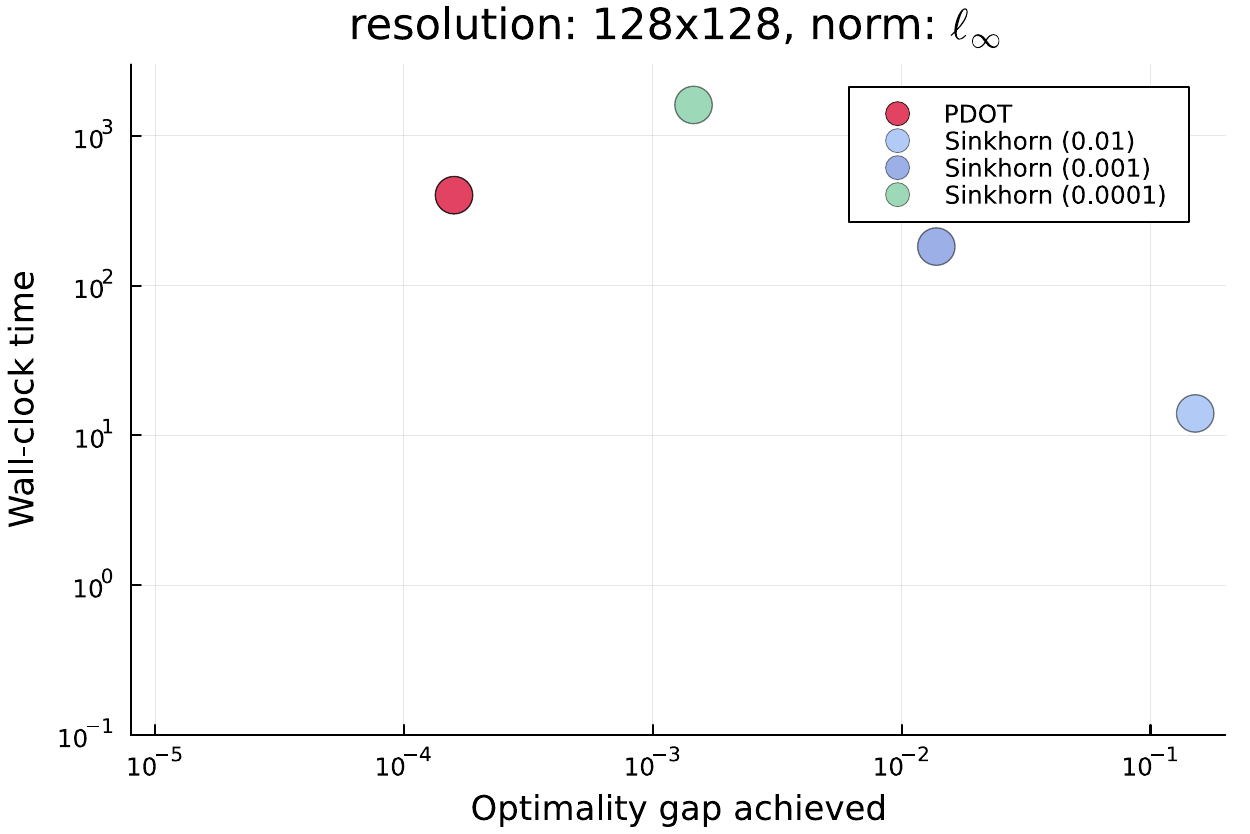}
	\end{tabular}
	\caption{Comparison of PDOT, Sinkhorn, and Gurobi.}
	\label{fig:time-gap}
\end{figure}
Figure \ref{fig:time-gap} compares PDOT with Sinkhorn and Gurobi on 4050 benchmark instances. We aggregate the results within each resolution and norm of cost matrix. The x-axis is the optimality gap achieved by the solution while the y-axis is the elapsed time in seconds. Each sub-figure is a specific combination of resolution and norm of cost matrix while each point represents a solver for solving OT instances. Particularly, Sinkhorn has three choices of penalty parameter $0.01$, $0.001$ and $0.0001$, and we set the tolerance of PDOT and Gurobi barrier method as $10^{-4}$. The dashed horizontal line is the time for network simplex in Gurobi to solve the problems to optimality (the default tolerance of Gurobi is $10^{-6}$). {For high resolution ($128\times 128$), Gurobi barrier and simplex fail for all instances with out-of-memory error, thus we do not report them in the figures.}
Figure \ref{fig:time-gap} yield several noteworthy observations:
\begin{itemize}
    \item PDOT demonstrates a competitive balance between wall-clock time and optimality gap across various test conditions. It notably exhibits lower gap values compared to Sinkhorn when allocated the same computation time. For instance, PDOT achieves an accuracy of $10^{-4}$ in approximately the same time it takes for Sinkhorn with a penalty of 0.001 to achieve a gap of $10^{-2}$. Furthermore, in the most demanding tests with a resolution of $128\times128$, Sinkhorn, at its lowest penalty setting, fails to solve all 450 instances within a one-hour limit. In contrast, PDOT resolves all instances to $10^{-4}$ accuracy within just 1000 seconds, showcasing its robustness and reliability in solving the optimal transport problem as defined in \eqref{eq:ot}.
    \item When compared against the Gurobi barrier, PDOT delivers solutions with comparable error levels but achieves this with a ten-fold speedup. The simplex method establishes an advantage in solving small instances with resolution $32\times 32$, in the sense that it finds the exact optimal solution with comparable time to PDOT. However, the simplex method faces scalability issues as demonstrated in solving instances with resolution $64\times 64$ -- PDOT is more than 10x faster than the simplex method. Additionally, it is notable that Gurobi (both barrier and simplex methods) runs out of memory when solving larger instances of $128\times 128$ resolution due to the extensive memory requirements of matrix factorization, which exceeds the simple storage needs for the instance data.
    \item The performance of the Sinkhorn algorithm is highly sensitive to the choice of penalty parameter. Selecting lower penalties (e.g., 0.0001) results in significantly longer computation times, though these are necessary to attain minimal gap values. The clear trade-off between computational time and solution quality is evident: Sinkhorn with a penalty of 0.01 is approximately 100 times faster than with a penalty of 0.0001, yet the latter can reach an accuracy of $10^{-3}$ while the former typically achieves gaps around $10^{-1}$.
\end{itemize}

\begin{table}[ht!]
\centering
{\scriptsize
\begin{tabular}{cccccccccc}
\hline
                    &                           & \multicolumn{2}{c}{\textbf{PDOT}}                  & \multicolumn{2}{c}{\textbf{\begin{tabular}[c]{@{}c@{}}Sinkhorn\\ (0.01)\end{tabular}}} & \multicolumn{2}{c}{\textbf{\begin{tabular}[c]{@{}c@{}}Sinkhorn\\ (0.001)\end{tabular}}} & \multicolumn{2}{c}{\textbf{\begin{tabular}[c]{@{}c@{}}Sinkhorn\\ (0.0001)\textsuperscript{\hyperlink{unsolved}{*}}\end{tabular}}}                \\
                    &                           & \texttt{time}           & \texttt{gap}              & \texttt{time}                  & \texttt{gap}                  & \texttt{time}                  & \texttt{gap}                   & \texttt{time}                   & \texttt{gap}               \\ \hline
\multirow{10}{*}{$\mathbf{\ell_1}$} & \texttt{CauchyDensity}    & 344.566                  & 0.000300                & 6.110                    & 0.162                   & 57.323                   & 0.0153                  & 632.372                  & 0.00153                 \\
                    & \texttt{ClassicImages}    & 563.336                  & 0.000194                & 12.378                   & 0.154                   & 201.252                  & 0.0137                  & 2198.295                 & 0.00138                 \\
                    & \texttt{GRFmoderate}      & 577.255                  & 0.000154                & 12.362                   & 0.155                   & 172.524                  & 0.0139                  & 2009.151                 & 0.00139                 \\
                    & \texttt{GRFrough}         & 919.888                  & 0.000187                & 16.229                   & 0.150                   & 336.061                  & 0.0125                  & 3418.149                 & 0.00128                 \\
                    & \texttt{GRFsmooth}        & 325.255                  & 0.000205                & 9.014                    & 0.161                   & 100.812                  & 0.0150                  & 1091.599                 & 0.00151                 \\
                    & \texttt{LogGRF}           & 236.010                  & 0.000205                & 6.046                    & 0.164                   & 58.199                   & 0.0158                  & 580.645                  & 0.00158                 \\
                    & \texttt{LogitGRF}         & 331.471                  & 0.000226                & 8.982                    & 0.156                   & 115.103                  & 0.0141                  & 1360.787                 & 0.00142                 \\
                    & \texttt{Microscopy} & 172.818                  & 0.000216                & 6.810                    & 0.146                   & 68.903                   & 0.0135                  & 781.662                  & 0.00135                 \\
                    & \texttt{Shapes}           & 119.582                  & 0.000164                & 3.834                    & 0.153                   & 31.467                   & 0.0140                  & 330.545                  & 0.00140                 \\
                    & \texttt{WhiteNoise}       & 681.050                  & 0.000351                & 14.118                   & 0.145                   & 591.640                  & 0.0109                  & 3600.027                 & 0.00124                 \\ \hdashline
                    \multirow{10}{*}{$\mathbf{\ell_2}$} & \texttt{CauchyDensity}    & 135.354         & 0.000397        & 6.455                 & 0.161                & 54.510                 & 0.0145               & 532.717                & 0.00135               \\
                    & \texttt{ClassicImages}    & 186.152         & 0.000281        & 12.463                & 0.154                & 150.252                & 0.0133               & 1578.028               & 0.00125               \\
                    & \texttt{GRFmoderate}      & 194.922         & 0.000197        & 12.454                & 0.155                & 154.271                & 0.0134               & 1615.615               & 0.00125               \\
                    & \texttt{GRFrough}         & 262.743         & 0.000174        & 19.135                & 0.149                & 325.212                & 0.0123               & 3219.705               & 0.00117               \\
                    & \texttt{GRFsmooth}        & 141.025         & 0.000404        & 9.302                 & 0.160                & 94.645                 & 0.0143               & 989.797                & 0.00134               \\
                    & \texttt{LogGRF}           & 119.140         & 0.000455        & 5.475                 & 0.163                & 46.821                 & 0.0148               & 460.166                & 0.00138               \\
                    & \texttt{LogitGRF}         & 134.225         & 0.000328        & 9.889                 & 0.156                & 105.627                & 0.0138               & 1080.814               & 0.00128               \\
                    & \texttt{Microscopy} & 120.759         & 0.000339        & 7.862                 & 0.146                & 59.656                 & 0.0129               & 594.133                & 0.00119               \\
                    & \texttt{Shapes}           & 82.127          & 0.000248        & 4.329                 & 0.152                & 32.663                 & 0.0135               & 303.963                & 0.00125               \\
                    & \texttt{WhiteNoise}       & 275.072         & 0.000282        & 18.079                & 0.143                & 729.127                & 0.0108               & 3600.027               & 0.00127               \\ \hdashline
                    \multirow{10}{*}{$\mathbf{\ell_\infty}$} & \texttt{CauchyDensity}    & 347.524        & 0.000361         & 8.902                 & 0.159                & 98.490                & 0.0152                & 1046.994               & 0.00153               \\
                      & \texttt{ClassicImages}    & 582.908        & 0.000196         & 17.211                & 0.151                & 236.919               & 0.0138                & 2535.618               & 0.00141               \\
                      & \texttt{GRFmoderate}      & 615.349        & 0.000179         & 19.042                & 0.152                & 261.632               & 0.0139                & 2765.548               & 0.00142               \\
                      & \texttt{GRFrough}         & 998.331        & 0.000185         & 25.689                & 0.145                & 514.443               & 0.0124                & 3554.389               & 0.00141               \\
                      & \texttt{GRFsmooth}        & 338.223        & 0.000271         & 13.461                & 0.159                & 159.348               & 0.0150                & 1617.257               & 0.00153               \\
                      & \texttt{LogGRF}           & 243.244        & 0.000234         & 9.0703                & 0.162                & 89.751                & 0.0156                & 802.649                & 0.00158               \\
                      & \texttt{LogitGRF}         & 371.306        & 0.000254         & 13.479                & 0.154                & 161.612               & 0.01438               & 1732.643               & 0.00145               \\
                      & \texttt{Microscopy} & 202.142        & 0.0001978        & 9.845                 & 0.144                & 96.024                & 0.0136                & 1068.548               & 0.00137               \\
                      & \texttt{Shapes}           & 145.248        & 0.000225         & 5.627                 & 0.151                & 51.059                & 0.0142                & 498.139                & 0.00142               \\
                      & \texttt{WhiteNoise}       & 925.432        & 0.000320         & 26.414                & 0.138                & 941.378               & 0.0107                & 3600.024               & 0.00149               \\ \hline
\end{tabular}
}
\caption{Detailed breakdown of the performance of PDOT and Sinkhorn (at various tolerance levels) across different norms of cost matrix under resolution $128\times 128$. Numbers in the parentheses stand for the selection of penalty parameter in Sinkhorn algorithm. \protect\textsuperscript{\protect\hypertarget{unsolved}{*}}Sinkhorn (0.0001) only solved 318 out of 450 instances within the 1-hour time limit. For failed instances, we utilize 3600s in computing the average running time, and the gap is computed at the termination solution.}
\label{tab:128}
\end{table}

Table \ref{tab:128} report the average results for each class with different norm of cost matrix under resolution $128\times 128$, in terms of time and gap. The results verify the observations summarized from Figure \ref{fig:time-gap} and demonstrate the consistency across different image classes. Moreover, it is noteworthy that Sinkhorn with most aggressive penalty parameter $0.0001$ cannot solve all the 450 instances within time limit (more precisely, it solves 318 out of 450 in total). This demonstrates it is challenging for Sinkhorn to provide high-quality solution efficiently. Results for medium-scale and small-scale instances are summarized in Table \ref{tab:64} and \ref{tab:32} that are deferred to Appendix \ref{app:tab}.

In summary, PDOT emerges as a particularly effective solver for optimal transport problems, offering a remarkable balance between speed and accuracy, making it well-suited for large-scale applications. Compared to the Sinkhorn algorithm, it can achieve higher accuracy with the same wall-clock time, and compared to LP solvers such as Gurobi, it is more scalable.

\section*{Acknowledgement}
The authors thank Yang (Justin) Meng for running preliminary numerical experiments.

\bibliographystyle{amsplain}
\bibliography{ref-papers}

\newpage
\appendix
\section{Tables}\label{app:tab}

\begin{table}[ht!]
\hspace{-2.2cm}
{\small
\begin{tabular}{ccccccccccccc}
\hline
                    &                           & \multicolumn{2}{c}{\textbf{PDOT}}                  & \multicolumn{2}{c}{\textbf{\begin{tabular}[c]{@{}c@{}}Sinkhorn\\ (0.01)\end{tabular}}} & \multicolumn{2}{c}{\textbf{\begin{tabular}[c]{@{}c@{}}Sinkhorn\\ (0.001)\end{tabular}}} & \multicolumn{2}{c}{\textbf{\begin{tabular}[c]{@{}c@{}}Sinkhorn\\ (0.0001)\end{tabular}}} & \multicolumn{2}{c}{\textbf{\begin{tabular}[c]{@{}c@{}}Gurobi\\ (barrier)\end{tabular}}}           &\textbf{\begin{tabular}[c]{@{}c@{}}Gurobi\\ (simplex)\end{tabular}}    \\
                    &                           & \texttt{time}           & \texttt{gap}              & \texttt{time}                  & \texttt{gap}                  & \texttt{time}                  & \texttt{gap}                   & \texttt{time}                   & \texttt{gap}   & \texttt{time}                   & \texttt{gap}       & \texttt{time}          \\ \hline
\multirow{10}{*}{$\mathbf{\ell_1}$} & \texttt{CauchyDensity}    & 14.941         & 0.000218         & 0.527                 & 0.134                & 6.356                 & 0.0127                & 61.525                 & 0.00127    & 112.160 & 3.78E-05 & 140.974      \\
                    & \texttt{ClassicImages}    & 23.441         & 0.000192         & 1.194                 & 0.126                & 18.443                & 0.0113                & 226.544                & 0.00114   & 139.758 & 4.45E-05 & 160.475         \\
                    & \texttt{GRFmoderate}      & 26.379         & 0.000147         & 1.141                 & 0.128                & 15.191                & 0.0116                & 178.200                & 0.00117     & 124.344 & 4.93E-05   & 165.234      \\
                    & \texttt{GRFrough}         & 31.857         & 0.000218         & 1.576                 & 0.122                & 26.529                & 0.0103                & 364.045                & 0.00104     & 131.186 & 5.88E-05     & 131.452   \\
                    & \texttt{GRFsmooth}        & 17.149         & 0.000172         & 0.745                 & 0.133                & 10.010                & 0.0123                & 108.477                & 0.00124   & 120.529 & 5.67E-05  & 169.657         \\
                    & \texttt{LogGRF}           & 13.053         & 0.000163         & 0.663                 & 0.132                & 8.807                 & 0.0124                & 98.464                 & 0.00125    & 133.167 & 3.95E-05   & 163.883     \\
                    & \texttt{LogitGRF}         & 16.715         & 0.000180         & 0.808                 & 0.129                & 10.538                & 0.0118                & 123.295                & 0.00119   & 134.784 & 5.05E-05  & 208.843      \\
                    & \texttt{Microscopy} & 10.690         & 0.000155         & 0.689                 & 0.116                & 7.690                 & 0.0109                & 76.824                 & 0.00110    & 33.504  & 4.96E-05  & 26.772     \\
                    & \texttt{Shapes}           & 6.185          & 0.000121         & 0.356                 & 0.125                & 3.330                 & 0.0115                & 28.932                 & 0.00115    & 39.515  & 2.69E-05  & 35.322         \\
                    & \texttt{WhiteNoise}       & 32.531         & 0.000277         & 1.484                 & 0.117                & 39.290                & 0.0094                & 547.665                & 0.00095   & 123.730 & 6.60E-05  & 247.829         \\ \hdashline
    \multirow{10}{*}{$\mathbf{\ell_2}$} & \texttt{CauchyDensity}    & 6.447                    & 0.000400                & 0.507                    & 0.133                   & 4.667                    & 0.0118                  & 45.613                   & 0.00110    & 104.047 & 7.37E-05 & 375.794     \\
                    & \texttt{ClassicImages}    & 8.792                    & 0.000298                & 1.120                    & 0.126                   & 13.400                   & 0.0108                  & 141.243                  & 0.00102 & 140.919 & 7.43E-05    & 307.649     \\
                    & \texttt{GRFmoderate}      & 8.994                    & 0.000270                & 1.007                    & 0.127                   & 12.201                   & 0.0110                  & 129.893                  & 0.00103  & 134.283 & 8.01E-05  & 490.623         \\
                    & \texttt{GRFrough}         & 11.239                   & 0.000242                & 1.848                    & 0.120                   & 26.148                   & 0.0010                  & 287.642                  & 0.00095        & 132.503 & 7.93E-05   & 186.218     \\
                    & \texttt{GRFsmooth}        & 7.154                    & 0.000321                & 0.733                    & 0.132                   & 8.405                    & 0.0116                  & 86.039                   & 0.00109         & 122.371 & 7.43E-05 & 565.480      \\
                    & \texttt{LogGRF}           & 6.664                    & 0.000409                & 0.665                    & 0.132                   & 6.363                    & 0.0117                  & 64.517                   & 0.00108 & 140.790 & 7.73E-05    & 328.094       \\
                    & \texttt{LogitGRF}         & 6.992                    & 0.000308                & 0.785                    & 0.128                   & 8.896                    & 0.0111                  & 91.368                   & 0.00103          & 130.695 & 7.54E-05  & 561.870     \\
                    & \texttt{Microscopy} & 6.672                    & 0.000269                & 0.720                    & 0.115                   & 7.325                    & 0.0101                  & 73.542                   & 0.00093  & 35.555  & 8.01E-05      & 27.994       \\
                    & \texttt{Shapes}           & 4.255                    & 0.000208                & 0.395                    & 0.124                   & 3.525                    & 0.0108                  & 29.879                   & 0.00100           & 42.535  & 7.06E-05  & 49.325 \\
                    & \texttt{WhiteNoise}       & 13.327                   & 0.000233                & 2.018                    & 0.115                   & 43.937                   & 0.00913                 & 487.518                  & 0.00089              & 127.682 & 7.43E-05  & 296.125 \\ \hdashline
                    \multirow{10}{*}{$\mathbf{\ell_\infty}$} & \texttt{CauchyDensity}    & 14.757                   & 0.000269                & 0.678                    & 0.131                   & 8.189                    & 0.0125                  & 86.007                   & 0.00126  & 109.872 & 4.47E-05     & 151.354     \\
                      & \texttt{ClassicImages}    & 25.751                   & 0.000193                & 1.489                    & 0.123                   & 22.275                   & 0.0114                  & 254.626                  & 0.00115           & 136.253 & 5.37E-05 & 183.669   \\
                      & \texttt{GRFmoderate}      & 27.400                   & 0.000176                & 1.465                    & 0.125                   & 22.347                   & 0.0116                  & 265.739                  & 0.00117          & 124.146 & 5.97E-05 & 196.212    \\
                      & \texttt{GRFrough}         & 37.051                   & 0.000206                & 2.432                    & 0.117                   & 42.172                   & 0.0102                  & 540.186                  & 0.00104         & 133.461 & 5.49E-05 & 139.808      \\
                      & \texttt{GRFsmooth}        & 17.307                   & 0.000285                & 1.102                    & 0.130                   & 15.366                   & 0.0124                  & 154.365                  & 0.00125            & 116.114 & 4.58E-05 & 171.876  \\
                      & \texttt{LogGRF}           & 13.004                   & 0.000295                & 0.960                    & 0.131                   & 11.044                   & 0.0126                  & 107.104                  & 0.00127          & 127.305 & 5.35E-05 & 163.617    \\
                      & \texttt{LogitGRF}         & 16.446                   & 0.000247                & 1.122                    & 0.125                   & 16.491                   & 0.0117                  & 216.312                  & 0.00119           & 134.200 & 4.99E-05 & 207.557     \\
                      & \texttt{Microscopy} & 11.790                   & 0.000167                & 1.000                    & 0.114                   & 11.877                   & 0.0108                  & 121.560                  & 0.00109             & 34.639  & 4.63E-05 & 26.952   \\
                      & \texttt{Shapes}           & 7.678                    & 0.000203                & 0.516                    & 0.123                   & 5.449                    & 0.0116                  & 42.469                   & 0.00117          & 38.624  & 2.99E-05   & 35.084    \\
                      & \texttt{WhiteNoise}       & 39.123                   & 0.000244                & 2.819                    & 0.111                   & 58.232                   & 0.0092                  & 725.494                  & 0.000942    & 127.640 & 5.64E-05  & 237.846          \\ \hline
\end{tabular}
}
\caption{Detailed breakdown of the performance of PDOT, Sinkhorn (at various tolerance levels), and Gurobi (barrier and simplex method) across different norms of cost matrix under resolution $64\times 64$. Numbers in the parentheses stand for the selection of penalty parameter in Sinkhorn algorithm.}
\label{tab:64}
\end{table}

\begin{table}[ht!]
\hspace{-1.8cm}
{\small
\begin{tabular}{ccccccccccccc}
\hline
                    &                           & \multicolumn{2}{c}{\textbf{PDOT}}                  & \multicolumn{2}{c}{\textbf{\begin{tabular}[c]{@{}c@{}}Sinkhorn\\ (0.01)\end{tabular}}} & \multicolumn{2}{c}{\textbf{\begin{tabular}[c]{@{}c@{}}Sinkhorn\\ (0.001)\end{tabular}}} & \multicolumn{2}{c}{\textbf{\begin{tabular}[c]{@{}c@{}}Sinkhorn\\ (0.0001)\end{tabular}}} & \multicolumn{2}{c}{\textbf{\begin{tabular}[c]{@{}c@{}}Gurobi\\ (barrier)\end{tabular}}}           & \textbf{\begin{tabular}[c]{@{}c@{}}Gurobi\\ (simplex)\end{tabular}}    \\
                    &                           & \texttt{time}           & \texttt{gap}              & \texttt{time}                  & \texttt{gap}                  & \texttt{time}                  & \texttt{gap}                   & \texttt{time}                   & \texttt{gap}   & \texttt{time}                   & \texttt{gap}       & \texttt{time}          \\ \hline
\multirow{10}{*}{$\mathbf{\ell_1}$} & \texttt{CauchyDensity}    & 1.892          & 0.000127         & 0.103                 & 0.107                & 1.203                 & 0.0104                & 10.143                & 0.00104  & 6.330 & 3.42E-05       & 2.583             \\
                    & \texttt{ClassicImages}    & 2.398          & 0.000194         & 0.346                 & 0.098                & 4.004                 & 0.0091                & 43.376                & 0.00092           & 7.548 & 3.16E-05      & 2.675   \\
                    & \texttt{GRFmoderate}      & 2.215          & 0.000205         & 0.208                 & 0.100                & 3.160                 & 0.0094                & 34.345                & 0.00095          & 7.427 & 4.44E-05   & 2.612    \\
                    & \texttt{GRFrough}         & 3.191          & 0.000185         & 0.600                 & 0.093                & 5.438                 & 0.0084                & 66.924                & 0.00085         & 8.103 & 4.39E-05  & 2.511       \\
                    & \texttt{GRFsmooth}        & 1.858          & 0.000130         & 0.127                 & 0.106                & 2.399                 & 0.0102                & 19.054                & 0.00102      & 6.836 & 4.58E-05  & 2.395          \\
                    & \texttt{LogGRF}           & 1.481          & 0.000157         & 0.097                 & 0.107                & 1.473                 & 0.0104                & 12.888                & 0.00104         & 7.347 & 3.69E-05 & 2.335          \\
                    & \texttt{LogitGRF}         & 1.841          & 0.000226         & 0.134                 & 0.100                & 2.324                 & 0.0094                & 25.351                & 0.00095    & 7.990 & 3.33E-05   & 2.586              \\
                    & \texttt{Microscopy} & 1.503          & 0.000124         & 0.111                 & 0.100                & 1.684                 & 0.0096                & 15.725                & 0.00096        & 4.143 & 3.66E-05   & 2.847          \\
                    & \texttt{Shapes}           & 0.978          & 0.000086         & 0.077                 & 0.097                & 0.973                 & 0.0091                & 7.396                 & 0.00092   & 1.775 & 3.08E-05  & 1.425              \\
                    & \texttt{WhiteNoise}       & 2.617          & 0.000326         & 0.579                 & 0.090                & 5.688                 & 0.0079                & 87.947                & 0.0008       & 7.640 & 4.34E-05   & 2.588        \\ \hdashline
    \multirow{10}{*}{$\mathbf{\ell_2}$} & \texttt{CauchyDensity}    & 1.003                    & 0.000279                & 0.093                    & 0.106                   & 0.870                    & 0.0094                  & 8.706                    & 0.00087     & 6.939 & 7.09E-05       & 3.026            \\
                    & \texttt{ClassicImages}    & 1.169                    & 0.000266                & 0.291                    & 0.097                   & 2.670                    & 0.0085                  & 28.862                   & 0.00081         & 8.642 & 7.31E-05   & 2.974            \\
                    & \texttt{GRFmoderate}      & 1.127                    & 0.000237                & 0.215                    & 0.099                   & 2.229                    & 0.0087                  & 25.149                   & 0.00082         & 8.813 & 7.61E-05 & 2.910           \\
                    & \texttt{GRFrough}         & 1.446                    & 0.000173                & 0.661                    & 0.092                   & 4.866                    & 0.0079                  & 55.631                   & 0.00077               & 9.012 & 7.81E-05  & 2.782  \\
                    & \texttt{GRFsmooth}        & 0.975                    & 0.000276                & 0.124                    & 0.104                   & 1.600                    & 0.0092                  & 16.899                   & 0.00086           & 7.892 & 6.68E-05 & 3.023       \\
                    & \texttt{LogGRF}           & 0.901                    & 0.000302                & 0.0898                   & 0.106                   & 0.855                    & 0.0094                  & 9.417                    & 0.00087            & 8.566 & 8.44E-05     & 2.937      \\
                    & \texttt{LogitGRF}         & 0.971                    & 0.000302                & 0.140                    & 0.100                   & 1.672                    & 0.0087                  & 18.638                   & 0.00082            & 8.583 & 7.45E-05 & 2.951     \\
                    & \texttt{Microscopy} & 0.934                    & 0.000258                & 0.114                    & 0.100                   & 1.253                    & 0.0087                  & 13.660                   & 0.00081             & 4.787 & 7.80E-05 & 3.013     \\
                    & \texttt{Shapes}           & 0.724                    & 0.000138                & 0.094                    & 0.096                   & 0.689                    & 0.0083                  & 6.185                    & 0.00077             & 1.940 & 6.36E-05  & 1.508     \\
                    & \texttt{WhiteNoise}       & 1.385                    & 0.000243                & 0.682                    & 0.088                   & 5.447                    & 0.0075                  & 70.713                   & 0.00076 & 8.421 & 6.78E-05    & 2.764            \\ \hdashline
    \multirow{10}{*}{$\mathbf{\ell_\infty}$} & \texttt{CauchyDensity}    & 2.133                    & 0.000162                & 0.116                    & 0.106                   & 1.541                    & 0.0104                  & 14.172                   & 0.00105     & 6.613 & 3.76E-05       & 2.395     \\
                      & \texttt{ClassicImages}    & 2.541                    & 0.000198                & 0.323                    & 0.095                   & 5.101                    & 0.0091                  & 55.918                   & 0.00093           & 7.621 & 4.37E-05 & 2.541     \\
                      & \texttt{GRFmoderate}      & 2.336                    & 0.000154                & 0.240                    & 0.097                   & 3.914                    & 0.0094                  & 41.214                   & 0.00095     & 7.727 & 3.75E-05 & 2.493            \\
                      & \texttt{GRFrough}         & 3.729                    & 0.000179                & 0.695                    & 0.088                   & 8.505                    & 0.0083                  & 96.955                   & 0.00084            & 8.203 & 4.62E-05 & 2.500     \\
                      & \texttt{GRFsmooth}        & 1.962                    & 0.000202                & 0.182                    & 0.103                   & 3.216                    & 0.0100                  & 31.678                   & 0.00101          & 7.125 & 4.30E-05 & 2.525       \\
                      & \texttt{LogGRF}           & 1.539                    & 0.000175                & 0.122                    & 0.106                   & 1.615                    & 0.0104                  & 15.427                   & 0.00105           & 7.376 & 4.44E-05  & 2.414     \\
                      & \texttt{LogitGRF}         & 1.801                    & 0.000232                & 0.200                    & 0.098                   & 3.947                    & 0.0095                  & 43.279                   & 0.00096          & 7.699 & 4.12E-05    & 2.558   \\
                      & \texttt{Microscopy} & 1.470                    & 0.000155                & 0.158                    & 0.099                   & 2.321                    & 0.0096                  & 22.898                   & 0.00097           & 4.201 & 4.12E-05 & 2.737       \\
                      & \texttt{Shapes}           & 1.019                    & 0.000120                & 0.108                    & 0.095                   & 1.143                    & 0.0092                  & 9.263                    & 0.00092           & 1.752 & 2.47E-05    & 1.386  \\
                      & \texttt{WhiteNoise}       & 3.320                    & 0.000264                & 0.766                    & 0.084                   & 9.376              & 0.0077                  & 124.207                  & 0.00079    & 7.856 & 4.28E-05  & 2.565     \\ \hline
\end{tabular}
}
\caption{Detailed breakdown of the performance of PDOT, Sinkhorn (at various tolerance levels), and Gurobi (barrier and simplex method) across different norms of cost matrix under resolution $32\times 32$. Numbers in the parentheses stand for the selection of penalty parameter in Sinkhorn algorithm.}
\label{tab:32}
\end{table}

\end{document}